\documentclass[11pt]{article}
\usepackage{epsfig}
\usepackage{amssymb,amsmath,amsthm,amscd}
\usepackage{latexsym}

\newtheorem{thm}{Theorem}[section]
\newtheorem{prop}[thm]{Proposition}

\newtheorem{lem}[thm]{Lemma}

\theoremstyle{definition}

\newcommand{\be}{\begin{equation}}
\newcommand{\ee}{\end{equation}}

\newcommand{\R}{\mathbb{R}}
\newcommand{\N}{\mathbb{N}}
\newcommand{\E}{\mathbb{E}}

\def \calm {{  {\mathcal{M}}  }}
\def \calo {{  {\mathcal{O}}  }}

 \def  \eps { { \varepsilon } }

\begin{document}

\baselineskip=1.2\baselineskip

\pagestyle{plain}
\title{ Long-time Behaviour  of the Non-autonomous Stochastic FitzHugh-Nagumo Systems
 on Thin Domains\footnote{This work was supported
 by NSFC (12371178) and
  Sichuan Science and Technology Program (2023ZYD0002).}}

\author{{ Dingshi Li, Ran Li\footnote{Corresponding authors:
liranzym@my.swjtu.edu.cn (R. Li).}, Tianhao Zeng}
 \\{ \small\textsl{School of Mathematics, Southwest Jiaotong University, Chengdu, 610031, P. R. China}}}

\date{}
\maketitle

{ \bf Abstract} \ \ \
We study the long-time behavior of non-autonomous stochastic FitzHugh-Nagumo systems
 on thin domains.  As  the  $(n+1)$-dimensional
thin domains collapses  onto an $n$-dimensional domain,
an  $n$-dimensional limiting FitzHugh-Nagumo system is derived.
This $n$-dimensional limiting system encodes the defining
geometry of the $(n+1)$-dimensional system.
To justify this limiting process, we
show that the pullback measure attractors of the FitzHugh-Nagumo systems
 on thin domains are upper semi-continuous
as the height of thin direction  tends to zero.



\section{Introduction}
\setcounter{equation}{0}

The FitzHugh-Nagumo system was proposed as a mathematical model to
describe the signal transmission across axons in neurobiology in \cite{Fit61,NAY62}.
The long-term behavior of this deterministic or stochastic system
 is extensively studied, such as attractors \cite{GW18, LY16, Wang09,Zhao21,Huang17,Wang2008},
 invariant manifolds \cite{ZS21}, and invariant measures \cite{WGW21}, etc.
Here we intend to investigate the long-time behavior of the non-autonomous stochastic
FitzHugh-Nagumo system with nonlinear noise.

Let $\mathcal Q$
be a smooth bounded domain in $  \mathbb{R}^n$
 and  $\mathcal O_\varepsilon$
 be the following $n+1$ dimensional region
 \[
{\mathcal{O}}_\varepsilon   = \left\{ {x=\left( {x^{\ast},x_{n+1}} \right)|x^{\ast}=(x_1,\ldots,x_n) \in \mathcal Q
\,\,\text{and}\,\,0 < x_{n+1} < \varepsilon g\left( x^{\ast} \right)} \right\},
\]
where
$0<\varepsilon\leq 1$
and $g\in C^2(\overline {\mathcal Q}, (0,+\infty))$ which implies that
  there exist
  positive constants $\underline{g}$ and $\overline{g}$ such that
\begin{equation}\label{a1}
\underline{g} \le g\left( {x^{\ast} } \right) \le\overline{g},\quad\forall x^{\ast}  \in \overline {\mathcal Q}.
\end{equation}
Denote $\mathcal{O}=\mathcal Q \times (0,1)$ and  $\widetilde{\mathcal O} = \mathcal Q \times [0,\overline{g}]$
 which contains $\mathcal O_\varepsilon$ for $0<\varepsilon\leq 1.$

In this paper, we
will  study the dynamical behavior of the stochastic
FitzHugh-Nagumo system  on
$\mathcal O_{\varepsilon}$ as $\varepsilon$ tends to zero:
\begin{equation}\label{eu1}
\left\{ \begin{array}{l}
d\hat{u}^\varepsilon\left( t \right) - \Delta \hat{u}^\varepsilon\left( t \right)dt+\alpha \hat{v}^\varepsilon\left( t \right)dt
= (F\left( {x,\hat{u}^\varepsilon\left( t \right)} \right)+G_1(t,x))dt\\
\quad\quad + \sum\limits_{k = 1}^\infty  {\left( {{{\sigma} _{1,k}}\left( x \right) + {\varpi _{1,k}}
 \left( {\hat{u}^\varepsilon\left( t \right)} \right)}
  \right)d{W_k}\left( t \right),\quad x \in {\mathcal O_\varepsilon },\,\
  t > \tau},  \\
  d\hat{v}^\varepsilon\left( t \right)+\gamma\hat{v}^\varepsilon\left( t \right)dt -\beta\hat{u}^\varepsilon\left( t \right)dt
= G_2(t,x)dt \\
\quad\quad+ \sum\limits_{k = 1}^\infty  {\left( {{{\sigma} _{2,k}}\left( x \right) + {\varpi _{2,k}}
\hat{v}^\varepsilon\left( t \right)}
  \right)d{W_k}\left( t \right),\quad x \in {\mathcal O_\varepsilon },\,\
  t > \tau},  \\
\frac{{\partial \hat{u}^\varepsilon}}{{\partial {\nu_\varepsilon }}} = 0,
\quad\frac{{\partial \hat{v}^\varepsilon}}{{\partial {\nu_\varepsilon }}} = 0,
\quad x \in \partial {\mathcal O_\varepsilon },
\end{array} \right.
\end{equation}
with initial data
\begin{equation}\label{eu2}
\hat{u}^\varepsilon \left( \tau,x \right) ={\hat \xi_1}^\varepsilon(x),
\quad \hat{v}^\varepsilon \left( \tau,x \right) ={\hat \xi_2}^\varepsilon(x),\quad x\in \mathcal O_\varepsilon,
\end{equation}
where $\tau\in \R$, $\nu_\varepsilon$ is the
unit outward  normal vector  to  $\partial \mathcal O_\varepsilon$,  $\alpha$, $\beta$ and $\gamma$ are positive constants,
$F$ is nonlinear function defined on
${\widetilde {\mathcal O}}\times \mathbb{R}$,
$G_1\in L^2_{loc}(\R, L^\infty( { {\widetilde {\mathcal O}} } ))$,
$G_2\in L^2_{loc}(\R, H^{1,\infty}( { {\widetilde {\mathcal O}} } ))$,
${\sigma} _{i,k}$  is function defined on ${\widetilde {\mathcal O}}$ for $i=1,2$,
$\varpi _{1,k}: \R\rightarrow \R$
is a  nonlinear function
for each   $k\in \N$, $\varpi _{2,k}$, $k\in \N$, are nonnegative constants,
and $(W_k)_{k\in \mathbb N}$ is a sequence of independent
standard Wiener processes on a complete filtered probability space $
( \Omega ,\mathcal F, \{\mathcal F_t  \}_{t \in\R} ,P )$  satisfying  the usual condition.

As $\varepsilon\rightarrow 0$,   as will show later,
the pullback measure attractors of \eqref{eu1}-\eqref{eu2}  converge to
the pullback measure attractors of the following system:
\begin{equation}{\label{eu3}}
\left\{ \begin{array}{l}
d{u^0}\left( t \right) - \frac{1}{g}\sum\limits_{i = 1}^n {{{\left( {gu_{{y_i}}^0} \right)}_{{y_i}}}dt+\alpha v^{0}(t)dt
 = F\left( {\left( {{y^*},0} \right),{u^0}} \right)dt} +G_1(t,\left( {{y^*},0} \right))dt\\
\quad\quad + \sum\limits_{k = 1}^\infty  {\left( {{\sigma _{1,k}}\left( {\left( {{y^*},0} \right)} \right)
+ {\varpi_{1,k}}\left( {{u^0}\left( t \right)}
 \right)} \right)d{W_k}\left( t \right),\quad {y^*}
 = \left( {{y_1}, \ldots ,{y_n}} \right) \in \mathcal Q,\,\,t > \tau,} \\
 d{v^0}\left( t \right)+\gamma v^0(t)dt-\beta u^{0}(t)dt =G_2(t,\left( {{y^*},0} \right))dt\\
\quad\quad + \sum\limits_{k = 1}^\infty  {\left( {{\sigma _{2,k}}\left( {\left( {{y^*},0} \right)} \right)
+ {\varpi_{2,k}}{v^0}\left( t \right)} \right)d{W_k}\left( t \right),\quad {y^*} \in \mathcal Q,\,\,t > \tau,} \\
\frac{{\partial u^0}}{{\partial \nu_0}} = 0,\quad\frac{{\partial v^0}}{{\partial \nu_0}} = 0,\quad y^* \in
\partial {\mathcal Q},
\end{array} \right.
\end{equation}
 with  initial condition
 \begin{equation}{\label{eu4}}
 u^0(\tau,y^{\ast})= \xi^0_1(y^{\ast}),\quad v^0(\tau,y^{\ast})= \xi^0_2(y^{\ast}),\quad y^{\ast}\in \mathcal Q.
 \end{equation}

 Hale and Rauge \cite{hale1992reaction, hale1992damped} first studied
 the limiting behavior of attractors of deterministic PDEs on thin domains.
  This subject is further studied by
  many experts since then, see, e.g.,
 \cite{hale1995reaction, arrieta2006dynamics,
  A1, A2,  ciuperca1996reaction, prizzi2001effect,
 raugel1993navier}.
Considering  the influence of random factors,
 the concept of random  attractors for random dynamical systems was introduced
in \cite{CDF97,CF94} to capture the dynamical behavior. The limiting behavior  of random  attractors for
systems    on thin domains was established in \cite{caraballo2007synchronization,
 li2018limiting,
  li2019limiting, li2017limiting, li2018regular, shi2019asymptotic} and the references therein.

Invariant measures
can be used to describe the  statistical dynamics of
random differential equations, which have been extensively
investigated in the literature,
see, e.g.,  \cite{brzezniak2017invariant,brzezniak2016invariant, eckmann2001invariant,
 kim2006invariant,wang2019dynamicsf,
wang2019dynamics,misiats2016existence,chen2021invariant}.
In particular, the
  limiting behavior of invariant measures
of stochastic equations
was   recently investigated in
\cite{li2021limiting, LWW21, chen2020limiting,liu22,liu2022,HJLY16}.
In the case that  the invariant measure is unique,
the limiting behavior  of invariant measures of the
stochastic   Navier-Stokes equations
defined in a three-dimensional
   thin domain $\mathbb{T}^2\times(0, \epsilon)$
   was investigated by
 Chueshov and Kuksin \cite{chueshov2008random, chueshov2008stochastic}
 where $\epsilon>0$ is a small parameter and
  $\mathbb{T}^2$ is a two-dimensional torus.

Measure attractors for   homogeneous Markov processes have been extensively studied in the literature,
see, e.g., \cite{sch91,M1992,CC98,sch97,cra08,sch99}. Just as the equilibrium point is contained
in the attractor, the invariant measure is contained in the measure attractor.
To describe the dynamic behavior of nonhomogeneous Markov processes,
Li and Wang \cite{LW24} introduced the concept of a pullback measure attractor,
whose structure is  characterized by the complete
    solutions of the non-autonomous dynamical system on measure space  generated by
 nonhomogeneous Markov processes. It should be emphasized that in \cite{DR06,DR07},
 the  complete  solution is referred to as an
 evolution systems of measures.

 In this paper,   we firstly
    prove the existence of pullback
 measure attractors of the stochastic equations on thin domains.
  Then, based on  the uniform estimates of solutions
  with respect to $\varepsilon$, we
    show the upper semi-continuity  of the pullback
 measure attractors of  \eqref{eu1}-\eqref{eu2} as  the   thin domains collapses.
To our knowledge, the FitzHugh-Nagumo system  is partially dissipative systems.
The partial dissipativity introduces a major
obstacle for examining the asymptotic compactness of dynamical system on measure space.
To overcome this obstacle,
we need to decompose
$\hat v^\varepsilon$ as a sum of two functions: one is regular in the sense that it belongs to $H^1(\calo)$ and the
other converges to zero in mean square moment as $t\rightarrow\infty$.
This splitting technique was used by several authors for the attractor problems
(see, for example \cite{Marion89,Wang2008}).

The outline of this paper is  as follows.
 In the next section, we
 recall the notation and theory of pullback
 measure attractors for
 non-autonomous dynamical
 systems  defined on the space of probability
 measures. In Section 3, we reformulate
  problem \eqref{eu1}-\eqref{eu2}
  and introduce a stochastic system
  defined on the fixed domain $\mathcal O$.
 the main results of this paper are presented in Section 4.
    Section 5  is devoted to the  uniform
    estimates of  solutions  for both  systems  \eqref{eu1}-\eqref{eu2}
    and \eqref{eu3}-\eqref{eu4}.
    In the last two sections, we prove the
     existence  and the convergence  of pullback  measure attractors of
    the  stochastic  equations, respectively.
    The letters $c_i$, $i\in \N,$ are generic positive constants which may change their values from
line to line or even in the same line.

\section{Preliminaries}
\setcounter{equation}{0}

In this section, for the reader's convenience,  we recall some results
regarding pullback measure attractors for non-autonomous
dynamical systems on the space of probability
measures(see, e.g., \cite{LW24}).
For details, please refer to the reference \cite{LW24}.

In what follows, we denote by $X$   a separable Banach space with norm  $\|\cdot\|_X$.  Let
$C_b(X)$  be  the space of bounded continuous functions $\varphi:X\rightarrow \R$
endowed with the norm
\[
\left\| \varphi \right\|_\infty   = \mathop {\sup }\limits_{x \in X} \left| {\varphi\left( x \right)} \right|,
\]
and $ C_b(X_w)$ as the space of bounded  weakly continuous functions.
Denote by $L_b(X)$ the space of bounded Lipschitz functions on $X$
which consists of all
  functions $\varphi\in C_b(X)$
  such that
\[
\text{Lip}\left( \varphi \right): = \mathop {\sup }\limits_{x_1 ,x_2  \in X,x_1\neq x_2} \frac{{\left| {\varphi \left( {x_1 } \right)}  -
{\varphi\left( {x_2 } \right)} \right|}}{{\|x_1-x_2\|_X}} < \infty .
\]
The space $L_b(X)$ is endowed with the norm
\[
\left\| \varphi \right\|_L  = \left\| \varphi \right\|_\infty   + \text{Lip}\left( \varphi \right).
\]
Let   $\mathcal P(X)$
be  the set of probability measures on $(X,\mathcal B(X))$, where
$\mathcal B(X)$ is the Borel $\sigma$-algebra of $X$.
Given   $\varphi\in C_b(X)$
and
$\mu \in  \mathcal P(X)$, we write
 $$
\left( { \varphi, \mu  } \right)
=  \int_{X} {\varphi \left( x \right)\mu \left( {dx} \right)}.
$$

Recall that
 a sequence $\{\mu_n \}_{n=1}^{\infty}
\subseteq   \mathcal P(X)$
is weakly convergent to
$\mu \in  \mathcal P(X)$
if
for every  $\varphi\in C_b(X)$,
$$
\lim_{n\to \infty }
  \left( { \varphi, \mu_n  } \right)
  =\left( { \varphi, \mu  } \right).
  $$
 Define a metric on $\mathcal P(X)$ by
\[
  d_{\mathcal P(X)} \left( {\mu _1 ,\mu _2 } \right) = \mathop {\sup }\limits_{\scriptstyle \varphi \in L_b \left( X \right) \hfill \atop
  \scriptstyle \left\| \varphi \right\|_L  \le 1 \hfill} \left| {\left( {\varphi,\mu _1 } \right) - \left( {\varphi,\mu _2 } \right)} \right|,
  \quad
  \forall \
  \mu_1,\mu_2\in \mathcal P(X).
\]
Then     $(\mathcal P(X), d_{\mathcal P(X)}  )$ is
a Polish space.
Moreover,    a sequence $\left\{ {\mu _n } \right\}_{n = 1}^\infty   \subseteq \mathcal P\left( X \right)$
converges to   $\mu$ in $(\mathcal P(X), d_{\mathcal P(X)}  )$  if and only if
 $\{\mu_n \}_{n=1}^{\infty}
 $
 converges to
$\mu  $ weakly.

Given $p >0$, let
$\mathcal P_p(X)$ be the
 subset of $\mathcal P(X)$
 as defined by
\[
\mathcal P_p \left( X \right) = \left\{ {\mu  \in \mathcal
P\left( X \right):
\int_X {\|x\|_X^p \mu \left( {dx} \right) <    \infty } } \right\}.
\]
Then
 $(\mathcal P_p(X), d_{\mathcal P(X)}  )$ is also a   metric space. Without confusion, we denote
 $(\mathcal P_p(X), d_{\mathcal P(X)}  )$ by  $(\mathcal P_p(X), d_{\mathcal P_p(X)}  )$.
   Given  $r>0$, denote by
 \[
 B_ {\mathcal P_p(X)} (r)  = \left\{ {\mu  \in \mathcal P_p \left( X \right)
 	:
 	\left(\int_X \|x\|_X^p \mu \left( {dx} \right)\right)^{\frac{1}{p}} \leq r } \right\}.
 \]
 Recall that
 the Hausdorff semi-metric
 between subsets of
 ${\mathcal P_p(X)}$ is
 given by
 $$
 d_ {\mathcal P_p(X)}(Y,Z) =
 \mathop {\sup }\limits_{y \in Y} \mathop {\inf }\limits_{z \in Z} d  \left( {y,z} \right),
 \quad
 Y, Z  \subseteq \mathcal P_p(X),
 \ \ Y, Z \neq \emptyset .
 $$

We now recall the following standard result.

\begin{prop}\label{Tesa}
	 Let $\mathcal D$ be a neighborhood-closed collection of
 families of subsets of $\mathcal P_p \left( X \right)$
 and
  $S$
 be  a
 continuous non-autonomous dynamical system on $\mathcal P_p \left( X \right)$.
 Then
 $S$ has a unique  $\mathcal D$-pullback measure attractor $\mathcal A$ in $\mathcal P_p \left( X \right)$
 if and only if $S$ has a
 closed  $\mathcal D$-pullback  absorbing
set $K\in \mathcal D$ and $S$ is
$\mathcal D$-pullback asymptotically  compact in $\mathcal P_p \left( X \right)$.
The $\mathcal D$-pullback measure attractor $\mathcal A$    is given by, for each $\tau\in \R$,
\begin{align*}
\begin{split}
\mathcal A\left( \tau \right)  = \omega(K,\tau)
 =\{\psi(0,\tau):\,\psi\,\text{is a}\,\,  \mathcal D \text{-complete orbit of}\,\,S\}\\
 =\{\xi(\tau):\,\xi\,\text{is a}\,\,  \mathcal D \text{-complete solution  of}\,\,S\}.
\end{split}
\end{align*}
If, in addition, both
$S$ and $K$ are $T$-periodic for some $T>0$, then so is the attractor $\mathcal A$,
i.e., $\mathcal A(\tau)=\mathcal A(\tau+T)$, for all $\tau\in \R$.
\end{prop}

\section{Abstract Formulation of Stochastic Equations}
\setcounter{equation}{0}

In this section, we discuss the assumptions on the
nonlinear terms   and reformulate  \eqref{eu1}
in the fixed domain $\mathcal{O}$.
Throughout this paper, we assume
  ${\sigma} _{i,k}\in L^\infty( { {\widetilde {\mathcal O}} } )$
  for  $i=1,2$ and every $k\in \N$ such that
   \begin{align}\label{cs1}
   \sum\limits_{k = 1}^\infty  {\left\| {\sigma _{i,k} } \right\|_{L^\infty
   \left( {\widetilde \calo} \right)}^2 }<\infty,
   \end{align}
   and
   \begin{align}\label{cs2}
   \left| {\frac{{\partial \sigma_{i,k}\left( {x} \right)}}{{\partial x}}}
\right| \le \sigma^{*}_{i,k} \left( x \right),
\end{align}
where $\sigma^{*}_{i,k} \in {L^\infty( { {\widetilde {\mathcal O}} } )}$
with
$
\sum\limits_{k = 1}^\infty  {\left\| {\sigma _{i,k}^* }
\right\|_{L^\infty  \left( {\widetilde \calo} \right)}^2 }  < \infty$.
For  every $ k \in \N$,
suppose   $\varpi_{1,k}: \R\rightarrow \R$  is
a  smooth function such that  for  all  $s_1,s_2\in {\mathbb{R}}$,
\begin{align}
& \left| {\varpi _{1,k} \left( s_1 \right)} \right| \le \beta _{k}  + \gamma _{k} \left| s_1 \right|, \label{cv1} \\
& \left| {\varpi _{1,k} \left( {s_1 } \right) - \varpi _{1,k} \left( {s_2 } \right)} \right|
\le L_{k} \left| {s_1  - s_2 } \right|, \label{cv3}
 \end{align}
where $\beta _{k}, \gamma _{k}$ and $L_{k}$ are positive constants satisfying
$\sum\limits_{k = 1}^\infty ( {\beta _{k}^2 }  + \gamma _{k}^2
  + L_{k}^2 ) < \infty.
$

 For the nonlinear drift term $F$, we assume that
  for all $x\in \widetilde{\mathcal O}$ and $s\in {\mathbb{R}}$,
\begin{align}
 &F\left( x,s \right)s \le  - \lambda_1  \left| s \right|^p +\varphi_1(x), \label{cf1}\\
  &\left| {F\left(x, s \right)} \right| \le \lambda _2
   \left| s \right|^{p - 1} +\varphi_2(x), \label{cf2}\\
& \frac{{\partial F\left( {x,s} \right)}}{{\partial s}} \le \varrho,\label{cf3}\\
&\left| {\frac{{\partial F\left( {x,s} \right)}}{{\partial x}}}
\right| \le \psi _3 \left( x \right),\label{cf4}
 \end{align}
 where $p\geq2$, $\lambda_1$, $\lambda_2$ and $\varrho$  are positive constants
 and $\varphi _1, \varphi _2,\psi _3
  \in {L^\infty( { {\widetilde {\mathcal O}} } )}.$
 In the sequel, we will  assume
 that
\begin{align}\label{La}
 \sum\limits_{k = 1}^\infty  {(\gamma _{k}^2+\varpi _{2,k}^2 })
 <{\frac 14} {\overline \lambda},
\end{align}
where $\overline \lambda =\min\{\lambda_1, \gamma\}$.   Then     we fix a positive number $\lambda$
such that
 \begin{align}\label{La1}
4 \sum\limits_{k = 1}^\infty  {(\gamma _{k}^2+\varpi _{2,k}^2) }
<\lambda< \overline\lambda,
 \end{align}
and set
\begin{equation}
\label{Ftof}
f(x,s) = F(x,s) +\lambda s
\end{equation}
 for all
$x\in \widetilde{\mathcal O}$ and $s\in {\mathbb{R}}$.
Then it follows   from \eqref{cf1}-\eqref{cf4} that
there exist positive numbers
$\alpha_1$, $\alpha_2$, $\varrho_1$,
$c_1$  and $c_2$   such that
\begin{align}
 &f\left( x,s \right)s \le  - \alpha_1  \left| s \right|^p +\varphi_1(x) +c_1, \label{c3.1a2}\\
  &\left| {f\left(x, s \right)} \right| \le \alpha _2 \left| s \right|^{p - 1} +\varphi_2(x)
  +c_2 , \label{c3.2a2}\\
& \frac{{\partial f\left( {x,s} \right)}}{{\partial s}} \le \varrho_1 ,\label{c3.3}\\
&\left| {\frac{{\partial f\left( {x,s} \right)}}{{\partial x}}} \right| \le \psi_3 \left( x \right).\label{c3.4}
 \end{align}
For convenience, we write $\psi_1(x) = \varphi_1(x) +c_1$
and  $\psi_2(x) = \varphi_2(x) +c_2$
 for $x\in \widetilde{\mathcal O}$ and $s\in {\mathbb{R}}$.

 With this notation,     \eqref{c3.1a2}
  and \eqref{c3.2a2} can be written as
\begin{align}
 &f\left( x,s \right)s \le  - \alpha_1  \left| s \right|^p +\psi_1(x) ,
  \label{c3.1}\\
  &\left| {f\left(x, s \right)} \right| \le \alpha _2 \left|
  s \right|^{p - 1} +\psi_2(x) . \label{c3.2}
 \end{align}

By   \eqref{eu1} and  \eqref{Ftof}  we get  for  $ t > \tau$,
\begin{equation}{\label{aehu2}}
\left\{ \begin{array}{l}
d\hat{u}^\varepsilon\left( t \right) -( \Delta \hat{u}^\varepsilon\left( t \right)-\lambda u^\varepsilon(t))dt
+\alpha \hat{v}^\varepsilon\left( t \right)dt
= (f\left( {x,\hat{u}^\varepsilon\left( t \right)} \right)+G_1(t,x))dt\\
\quad\quad + \sum\limits_{k = 1}^\infty  {\left( {{{\sigma} _{1,k}}\left( x \right) + {\varpi _{1,k}}
 \left( {\hat{u}^\varepsilon\left( t \right)} \right)}
  \right)d{W_k}\left( t \right),\quad x \in {\mathcal O_\varepsilon }},  \\
  d\hat{v}^\varepsilon\left( t \right)+\gamma\hat{v}^\varepsilon\left( t \right)dt -\beta\hat{u}^\varepsilon\left( t \right)dt
= G_2(t,x)dt \\
\quad\quad+ \sum\limits_{k = 1}^\infty  {\left( {{{\sigma} _{2,k}}\left( x \right) + {\varpi _{2,k}}
{\hat{v}^\varepsilon\left( t \right)}}
  \right)d{W_k}\left( t \right),\quad x \in {\mathcal O_\varepsilon }},  \\
\frac{{\partial \hat{u}^\varepsilon}}{{\partial {\nu_\varepsilon }}} = 0,
\quad\frac{{\partial \hat{v}^\varepsilon}}{{\partial {\nu_\varepsilon }}} = 0,
\quad x \in \partial {\mathcal O_\varepsilon },
\end{array} \right.
\end{equation}
 with  initial condition
 \begin{equation}{\label{aehu3}}
\hat u^\varepsilon(\tau,x)={\hat\xi}^\varepsilon_1(x),\quad 
\hat v^\varepsilon(\tau,x)={\hat\xi}^\varepsilon_2(x),\quad x\in \mathcal O_\varepsilon.
 \end{equation}

We  now  transfer  problem  \eqref{aehu2}-\eqref{aehu3}  into
a boundary value problem  in  the fixed domain $\mathcal O$.
As usual,   define
$T_\varepsilon: \mathcal O_\varepsilon \rightarrow \mathcal O$
by
$T_\varepsilon (x^{\ast}, x_{n+1})
= \left (x^{\ast}, {\frac {x_{n+1}}{\varepsilon g(x^{\ast})}} \right )
$ for $x= (x^{\ast}, x_{n+1}) \in
 \mathcal O_\varepsilon$. Let $y=(y^{\ast}, y_{n+1}) =
T_\varepsilon (x^{\ast}, x_{n+1})$. Then we  have
\[
x^{\ast} = y^{\ast}  ,\quad\quad x_{n+1}  = \varepsilon g\left( {y^{\ast} } \right)y_{n+1}.
\]

It follows from \cite{liu2010poisson}
 that  the  Laplace operator  in the original
variable $x\in {\mathcal{O}}_\varepsilon $ and  in the new variable
$y\in \mathcal{O}$
is related by
\[
\Delta_x \hat u(x)
=  \frac{1}{g}\text{div}_y ( P_\varepsilon  u(y)),
\]
 where we denote by  $\hat u(x)=u(y)$
and $P_\varepsilon $ is the operator  given by
\[
P_\varepsilon  u (y)    = \left ( {\begin{array}{*{20}c}
   {gu_{y_1 }   - g_{y_1 } y_{n + 1} u_{y_{n + 1} }  }  \\
    \vdots   \\
   {gu_{y_n }  - g_{y_n } y_{n + 1} u_{y_{n + 1} }  }  \\
   { - \sum\limits_{i = 1}^n {y_{n + 1} g_{y_i } u_{y_i }  }  + \frac{1}{{\varepsilon
   ^2 g}}\left( {1 + \sum\limits_{i = 1}^n {\left( {\varepsilon y_{n + 1} g_{y_i } }
   \right)^2 } } \right)u_{y_{n + 1} }  }  \\
\end{array}} \right ).
\]

For  $y = \left( {y^{\ast}  ,y_{n+1} } \right) \in \mathcal O$ and $s\in {\mathbb{R}}$,
we introduce

 \[
F_\varepsilon  \left( {y^{\ast}  ,y_{n+1} },s \right) = F\left(
 {y^{\ast} ,\varepsilon g\left( y^{\ast}  \right)y_{n+1} }, s \right),
 \qquad  F_0 \left( y^{\ast}, s\right) = F\left({y^{\ast} ,0},s \right),
\]
\[
f_\varepsilon  \left( {y^{\ast}  ,y_{n+1} },s \right) = f\left(
 {y^{\ast} ,\varepsilon g\left( y^{\ast}  \right)y_{n+1} },s \right),
 \qquad  f_0 \left( y^{\ast}, s\right) = f\left( {y^{\ast} ,0},s \right),
\]
\[G_{i,\varepsilon}  \left( {t,y^{\ast}  ,y_{n+1} } \right) = G_i\left(
 {t,y^{\ast} ,\varepsilon g\left( y^{\ast}  \right)y_{n+1} } \right),
 \qquad  G_{i,0} \left( t,y^{\ast}\right) = G_i\left( {t,y^{\ast} ,0}\right),
\]
and
\[
\sigma_{i,k,\varepsilon}  \left( {y^{\ast} ,y_{n+1} } \right) = \sigma_{i,k} \left(
 {y^{\ast} ,\varepsilon g\left( y^{\ast} \right)y_{n+1} } \right),
 \qquad \sigma_{i,k,0}  \left( {y^{\ast} } \right) = \sigma_{i,k} \left( {y^{\ast}, 0} \right).
 \]
Then   problem \eqref{aehu2}-\eqref{aehu3}  is equivalent to the following  system
\begin{equation}{\label{aeu1}}
\left\{ \begin{array}{l}
 du^\varepsilon -  ( {\frac 1g} \text{div}_y ( P_\varepsilon  u^\varepsilon ) -\lambda u^\varepsilon)dt+\alpha v^\varepsilon(t)dt
 = f_\varepsilon\left(y, u^\varepsilon \right) dt +G_{1,\varepsilon}(t,y)dt\\
 \quad \quad + \sum\limits_{k = 1}^\infty  {\left( {{\sigma _{1,k,\varepsilon}}\left( y \right) +
{\varpi_{1,k}}\left( {{u}^\varepsilon\left( t \right)} \right)}
   \right)d{W_k}\left( t \right),\quad y = \left( {y^{\ast}  ,y_{n+1} } \right) \in \mathcal O,\quad t > \tau,} \\
   d{v}^\varepsilon\left( t \right)+\gamma{v}^\varepsilon\left( t \right)dt -\beta{u}^\varepsilon\left( t \right)dt
= G_{2,\varepsilon}(t,y)dt \\
\quad\quad+ \sum\limits_{k = 1}^\infty  {\left( {{{\sigma} _{2,k,\varepsilon}}\left( y \right) + {\varpi _{2,k}}
{{v}^\varepsilon\left( t \right)}}
  \right)d{W_k}\left( t \right),\quad y  \in \mathcal O,\quad   t > \tau},  \\
  P_\varepsilon  u^\varepsilon \cdot \nu = 0,\quad P_\varepsilon  v^\varepsilon \cdot \nu = 0,\quad y\in\partial \mathcal O , \\
 \end{array} \right.
\end{equation}
with  initial condition
\begin{equation}{\label{aeu2}}
u^\varepsilon(\tau,y)=\xi^\varepsilon_1(y)
={\hat \xi}^\varepsilon_1(T^{-1}_\varepsilon(y)),\quad v^\varepsilon(\tau,y)=\xi^\varepsilon_2(y)
={\hat \xi}^\varepsilon_2(T^{-1}_\varepsilon(y)),\quad y\in \mathcal O,
\end{equation}
where $\nu$ is the unit outward normal vector  to $\partial \mathcal O$.

Now we want to write
 equation  (\ref{aeu1})  as  an  abstract evolutionary equation.
 We first introduce  the inner product $\left( { \cdot , \cdot } \right)_{H_g(\mathcal O) }$
 on $L^2(\mathcal O)$ defined by
\[
\left( {u,v} \right)_{H_g(\mathcal O) }  = \int_{\mathcal O}
 {guvdy},\quad \text{for all}\,\, u,v\in L^2(\mathcal O)
\]
and denote by $H_g(\mathcal O)$ the space equipped with this inner product.
Since $g$  is a continuous function on $\overline{\mathcal Q}$ and satisfies (\ref{a1}),
one  can easily show  that $H_g(\mathcal O)$  is a Hilbert space with  norm
equivalent to the natural norm of $L^2(\mathcal O)$.

For  $0<\varepsilon\leq1$,
 we  introduce  a
 bilinear form $a_\varepsilon  \left( { \cdot , \cdot } \right)
$:$\,\,H^1 \left( \mathcal O  \right) \times H^1 \left( \mathcal O \right) \to {\mathbb{R}}$,
 given by
\begin{eqnarray}\label{a4.1}
a_\varepsilon  \left( {u,v} \right) =
\left (J^* \nabla_y u, \  J^*\nabla_y v \right ) _{H_g \left( \mathcal O \right)}
\quad \mbox{for} \ u, \ v \in H^1 \left( \mathcal O  \right),
 \end{eqnarray}
 where
 $$
 J^* \nabla_y u (y)
 =     \left( {u_{y_1 }  - \frac{{g_{y_1 } }}
{g}y_{n + 1} u_{y_{n + 1} } , \ldots ,u_{y_n }  - \frac{{g_{y_n } }}
{g}y_{n + 1} u_{y_{n + 1} } ,\frac{1}{{\varepsilon g}}u_{{y_{n + 1} } } } \right).
$$
Let $H_\varepsilon ^1(\mathcal O)$ be the space $H^1(\mathcal O)$ endowed with the norm
\begin{equation}\label{a4.2}
\| u \|_{H_\varepsilon ^1(\mathcal O) }  = \left( \|u\|^2_{H^1(\mathcal O)}    + \frac{1}{\varepsilon ^2 }
\| u _{y_{n+1}}\|^2_{L^2 (\mathcal O)}  \right)^{\frac{1}{2}}.
\end{equation}
We see that there exist positive constants
$\varepsilon_0$,  $\eta_1$ and $\eta_2$ such that
for all $0<\varepsilon<\varepsilon_0$ and   $u\in H^1(\mathcal O),$
\begin{equation}\label{4.1}
\eta_1\left\| u \right\|_{H_\varepsilon ^1(\mathcal O) }^2  \le a_\varepsilon  \left( {u,u} \right)
 + \left\| u \right\|_{L^2({\mathcal O})}^2 \le \eta_2\left\| u \right\|_{H_\varepsilon ^1(\mathcal O) }^2.
\end{equation}
Hereafter,  we fix the number
 $\varepsilon_0>0$
 such that \eqref{4.1} is fulfilled for all $\varepsilon \in (0, \varepsilon_0)$.

Denote by  $A_\varepsilon$
the  unbounded operator on $H_g(\mathcal O)$ with domain
$
D\left( {A_\varepsilon  } \right) = \left\{ {u \in H^2 \left( \mathcal O  \right),
P_\varepsilon  u  \cdot \nu = 0\,\, \text{on}\,\, \partial \mathcal O } \right\}
$  as  defined by
\[
A_\varepsilon  u =  - \frac{1}{g}\text{div}_y (P_\varepsilon  u),\quad  u\in D\left( {A_\varepsilon  } \right).
\]
Then we  have
\begin{equation}\label{aae4.1}
a_\varepsilon  \left( {u,v} \right) = \left( {A_\varepsilon  u,v} \right)_{H_g({\mathcal O}) },\quad
\forall \;u \in D\left( {A_\varepsilon  } \right),\forall\; v \in H^1(\mathcal O).
\end{equation}
By   $ A_\varepsilon$,
 problem    (\ref{aeu1})-(\ref{aeu2}) can be written  as
\begin{equation}\label{aeu3}
\left\{ \begin{array}{l}
 du^\varepsilon + (A_\varepsilon  u^\varepsilon+ \lambda u^\varepsilon)dt+\alpha v^\varepsilon (t)dt
 =  {f_\varepsilon\left(y, {u^\varepsilon} \right)}dt+G_{1,\varepsilon}(t,y)dt\\
  \quad \quad + \sum\limits_{k = 1}^\infty  {\left( {{\sigma _{1,k,\varepsilon}}\left( y\right) +
 {\varpi _{1,k}}\left( {{u}^\varepsilon\left( t \right)} \right)}
   \right)d{W_k}\left( t \right),\quad t > \tau,} \\
      d{v}^\varepsilon\left( t \right)+\gamma{v}^\varepsilon\left( t \right)dt -\beta{u}^\varepsilon\left( t \right)dt
= G_{2,\varepsilon}(t,y)dt \\
\quad\quad+ \sum\limits_{k = 1}^\infty  {\left( {{{\sigma} _{2,k,\varepsilon}}\left( y \right) + {\varpi _{2,k}}
{{v}^\varepsilon\left( t \right)}}
  \right)d{W_k}\left( t \right),\quad   t > \tau},  \\
u^\varepsilon\left( {\tau} \right) =\xi^\varepsilon_1,\quad v^\varepsilon\left( {\tau} \right) =\xi^\varepsilon_2.  \\
 \end{array} \right.
\end{equation}

To reformulate   system (\ref{eu3})-(\ref{eu4}), we
  introduce  an  inner product $\left( { \cdot , \cdot } \right)_{H_g(\mathcal Q) }$
 on $L^2(\mathcal Q)$  as defined by
\[
\left( {u,v} \right)_{H_g(\mathcal Q) }
 = \int_{\mathcal Q}  {guvdy^{\ast} },\quad \text{for all}\,\, u,v\in L^2(\mathcal Q),
\]
and denote by $H_g(\mathcal Q)$ the space equipped with this inner product.
Let     $a_0 \left( { \cdot , \cdot } \right)
$:$\,\,H^1 \left( \mathcal Q\right) \times H^1 \left( \mathcal Q \right) \to {\mathbb{R}}$
be   		a bilinear form
 given by
\[
a_0 \left( {u,v} \right) = \int_{\mathcal Q} {g\nabla}  u\cdot\nabla v dy^{\ast}.
\]
Denote by $A_0$  the
  unbounded operator  on $H_g(\mathcal Q)$ with domain
$
D\left( {A_0  } \right) = \left\{ {u  \in H^2 \left( \mathcal Q  \right),
 \frac{\partial u}{\partial \nu_0} = 0\,\, \text{on}\,\, \partial \mathcal Q} \right\}
$ as  defined by
\[
A_0  u =  - \frac{1}{g}
\sum\limits_{i=1}^{n} {(gu_{y_i})_{y_i}}\quad u \in D\left( {A_0  } \right).
\]
Then we  have
\[
a_0  \left( {u,v} \right) = \left( {A_0  u,v} \right)_{H_g(\mathcal Q) } ,
\quad\forall\; u \in D\left( {A_0 } \right), \forall\; v \in H^1(\mathcal Q).
\]
By  $A_0$,    system (\ref{eu3})-(\ref{eu4}) can be written as
\begin{equation}\label{aeu4}
\left\{ \begin{array}{l}
 du^0 + (A_0  u^0 + \lambda u^0)dt +\alpha v^0(t)dt=   f_0\left(y^{\ast}, u^0
  \right) dt+G_{1,0}(t,y^\ast) dt\\
  \quad + \sum\limits_{k = 1}^\infty  {\left( {{\sigma _{1,k,0}}\left( y^*  \right) +
  {\varpi _{1,k}}\left( {{u}^0\left( t \right)} \right)}
   \right)d{W_k}\left( t \right),\quad t > \tau,} \\
      d{v}^0\left( t \right)+\gamma{v}^0\left( t \right)dt -\beta{u}^0\left( t \right)dt
= G_{2,0}(t,y^{\ast})dt \\
\quad\quad+ \sum\limits_{k = 1}^\infty  {\left( {{{\sigma} _{2,k,0}}\left( y \right) + {\varpi _{2,k}}
{{v}^0\left( t \right)}}
  \right)d{W_k}\left( t \right),\quad   t > \tau},  \\
 u^0\left( {\tau}\right) = \xi^0_1,\quad  v^0\left( {\tau}\right) = \xi^0_2.  \\
 \end{array} \right.
\end{equation}

Our main purpose  of the paper is to prove that systems \eqref{aeu3}
  and \eqref{aeu4} possess  pullback
measure attractors $\{\mathcal A_\varepsilon(\tau)\}_{\tau \in \R}$ in $L^2(\mathcal {O})\times L^2(\mathcal {O})$
   and   pullback
measure attractors $\{\mathcal A_0(\tau)\}_{\tau \in \R}$
 in   $L^2(\mathcal Q)\times L^2(\mathcal Q)$, respectively.
 Furthermore,  we will study  the upper-continuous of
 $\mathcal A_\varepsilon$   as $\varepsilon \to 0$.

For simplicity,   hereafter, we write $\mathbb L^2(\calo)=L^2(\calo)\times L^2(\calo)$,
  $\mathbb H_g(\calo)=H_g(\calo)\times H_g(\calo)$ and $\mathbb H^1_\varepsilon(\calo)
  =H^1_\varepsilon(\calo)\times H^1_\varepsilon(\calo)$.
 $\mathbb L^2(\mathcal Q)$, $\mathbb H_g(\mathcal Q)$ and $\mathbb H^1_\varepsilon(\mathcal Q)$ is defined  by analogy.
For a Banach space $X$ and $\tau\in \R$, we  use
 $L^2_{\mathcal{F}_\tau}(X) $
 to denote the space of all ${\mathcal F_\tau}$-measurable,
$X$-valued random variables $\varphi$
with $\E\|\varphi\|_{X}^2<\infty$,
where $\mathbb E$  means   the mathematical expectation. By  the arguments of  \cite[Theorem 6.1]{WGW21},
one can verify that if
\eqref{cs1}-\eqref{cf4} hold,
then  for any $\xi^\varepsilon=(\xi^\varepsilon_1, \xi^\varepsilon_2)  \in
L^2_{\mathcal{F}_\tau}( \mathbb L^2(\mathcal {O}))$ (respectively, $\xi^0=(\xi^0_1, \xi^0_2)  \in
L^2_{\mathcal{F}_\tau}( \mathbb L^2(\mathcal {Q}) ))$,
system \eqref{aeu3} (respectively,  \eqref{aeu4}) has a unique solution
  in the   sense of Definition 2.1 in \cite{WGW21}, which is
  denoted by   $w^\varepsilon(t,\tau,\xi^\varepsilon)=
  (u^\varepsilon(t,\tau,\xi^\varepsilon),v^\varepsilon(t,\tau,\xi^\varepsilon))$
  or $w^\varepsilon(t)=(u^\varepsilon(t), v^\varepsilon(t))$
  (respectively, $w^0(t,\tau,\xi^0)=(u^0(t,\tau,\xi^0),v^0(t,\tau,\xi^0))$ or $w^0(t)=(u^0_1(t),v^0_1(t))$).

For convenience, we write
  $Z_0=\mathbb L^2(\mathcal Q)$ and $Z_1=\mathbb L^2(\mathcal O)$.
  Give a   subset
  $E$ of $\mathcal P_2 \left( Z_i \right)$
  for $i=1,2$,  define
 \[
\|E  \|_{\mathcal P_2 \left( Z_i \right)}
=\inf \left\{ {r > 0:
	\sup_{\mu\in E}
	\left (
	\int_{Z_i} \|z \|_{Z_i}^2
	\mu (dz)\right ) ^{\frac 12}
	\le r
  }
	   \right\},
\]
with the convention that
inf $\emptyset =\infty$.
If  $E$ is
 a  bounded
subset of  $\mathcal P_2 \left( Z_i \right)$, then
$
\|E  \|_{\mathcal P_2 \left( Z_i \right)}<\infty
$.
  For  $i= 0, 1$, let  $\mathcal{D}_i$ be the collection of   families  of   bounded nonempty subsets
 of $\mathcal P_2 \left( Z_i \right)$ as given by
\[
\mathcal D_i = \left\{ {D
	 = \left\{ {D\left( \tau  \right) \subseteq \mathcal P_2 \left( Z_i \right):
\emptyset\ne  D\left( \tau  \right) \,\, \text{bounded in }\,\,\mathcal P_2 \left( Z_i \right),\,\,\tau  \in \R} \right\}:\,\,
\mathop {\lim }\limits_{\tau  \to  - \infty } e^{\zeta \tau }
\left\| {D\left( \tau  \right)} \right\|_{\mathcal P_2 (Z_i )}^2  = 0
} \right\},
\]
where  $\zeta=\lambda  - 4\sum\limits_{k = 1}^\infty  ({\gamma _k^2+\varpi_{2,k}^2})>0$
by  \eqref{La1}.

Throughout this paper, we assume
\begin{align}\label{acg}
\int_{ - \infty }^\tau  {e^{\zeta t}( \left\| {G_1\left( t,\cdot \right)} \right\|_{L^\infty( { {\widetilde {\mathcal O}} } )}^2
+\left\| {G_2\left( t,\cdot \right)} \right\|^2_{H^{1,\infty}( { {\widetilde {\mathcal O}} } )}) }dt
< \infty ,\quad\forall \tau  \in \R.
\end{align}

\section{Main Results}
\setcounter{equation}{0}

In this subsection, we present the main results of our paper.

As usual, if $\varphi :\mathbb L^2(\calo) \to \R
$ is a bounded Borel function,
then for $ r\leq t$ and $\xi=(\xi_1,\xi_2)  \in \mathbb L^2(\calo)$,   we  set
\[
\left( {p^\varepsilon({t,r}) \varphi } \right)\left( \xi  \right)
= \E\left( {\varphi \left( {u^\varepsilon\left( {t,r,\xi_1 } \right), v^\varepsilon\left( {t,r,\xi_2 } \right)  } \right)} \right)
\]
and
\[
P^\varepsilon\left( {r,\xi ;t,\Gamma } \right) = \left( {p^\varepsilon(t,r) 1_\Gamma  } \right)
\left( \xi  \right)  ,
\]
where  $
\Gamma  \in \mathcal B\left( \mathbb L^2(\calo)\right)$
 and
 $1_{\Gamma}$ is the characteristic function of $\Gamma$.

The following
 results
are standard  and the proof is  omitted.

\begin{lem}\label{markovp1}
 If \eqref{cs1}-\eqref{La} hold, then

(i) The family $\left\{ {p^\varepsilon({t,r}) } \right\}_{ r \le t} $ is Feller;
 that is, for any ${ r \le t}$, the function $
p^\varepsilon({t,r}) \varphi\in C_b(L^2(\calo))
$  if    $\varphi \in C_b(L^2(\calo))$.

(ii)   For every $r
\in \R$ and $
\xi^\varepsilon  \in \mathbb L^2(\calo)$, the process $
\left\{ {u^\varepsilon \left( {t,r,\xi^\varepsilon } \right),v^\varepsilon \left( {t,r,\xi^\varepsilon } \right)} \right\}_{t \ge r}$
 is a $\mathbb L^2(\calo)$-valued Markov
process.
\end{lem}

Given $t\geq r$ and $\mu \in \mathcal P(\mathbb L^2(\calo))$, define
\[
p_ * ^\varepsilon  \left( {t,r } \right)\mu \left(  \cdot  \right) =
\int_{\mathbb L^2(\calo)} {P^\varepsilon\left( {r ,\xi; t, \cdot} \right)} \mu \left( {d\xi} \right).
\]
Then  $p_{\ast}^\varepsilon(t,r): \mathcal P(\mathbb L^2(\calo))\rightarrow \mathcal P(\mathbb L^2(\calo))$ is the dual
operator of $p^\varepsilon({t,r})$.
By the definition of the solution of  \eqref{aeu3} we find that
  for all $t\geq r $,
$p_{\ast}^\varepsilon(t, r)$
maps
$ \mathcal P_2(\mathbb L^2(\calo))$
to $ \mathcal P_2(\mathbb L^2(\calo))$.

Given   $t\geq r$,
the operator
$p_{\ast}^0(t, r): \mathcal P_2(\mathbb L^2(\mathcal Q))\rightarrow \mathcal P_2(\mathbb L^2(\mathcal Q))$ associated with \eqref{aeu4}
can be  defined in the same
fashion
as
$p_{\ast}^\varepsilon(t, r)$.

We will also investigate the periodicity of pullback   measure attractors  of
system \eqref{eu3}-\eqref{eu4} for which we assume that
$G_1$ and $G_2$  are $\rho$-periodic in $t$ for some $\rho>0$.

By the similar argument as that of Lemma 4.1 in \cite{LWW21}, we get the
following lemma.
\begin{lem}\label{markovp2}
 Suppose  \eqref{cs1}-\eqref{La} hold. Then we have
the family $\left\{ {p^\varepsilon({t,r}) } \right\}_{  r \le t} $, $0<\varepsilon<\varepsilon_0$,
is $\rho$-periodic; that is,  for all $ t\geq r$,
\[
P^\varepsilon\left( {r,\xi ;t, \cdot } \right) = P^\varepsilon\left( {r+\rho,\xi ;t+\rho, \cdot } \right),
\quad\forall \xi  \in \mathbb L^2(\calo).
\]
\end{lem}

 We now define
 a non-autonomous dynamical system
 $S^\varepsilon(t,\tau)$,
 $t\ge \tau$,
 for the family
 of operators
 $p_{\ast}^\varepsilon(t, \tau )$.
 Given
  $t\in \R^+$, $\tau\in \R$ and
  $0<\varepsilon<\varepsilon_0$,
  let $S^\varepsilon(t,\tau):\mathcal P_2(\mathbb L^2(\calo))\rightarrow\mathcal P_2(\mathbb L^2(\calo))$
  be the map given by
\be \label{sp}
S^\varepsilon(t,\tau)\mu
=p_{\ast}^\varepsilon(\tau+t,\tau)
\mu ,
\quad \forall  \
 \mu \in \mathcal P_2(\mathbb L^2(\calo)).
\ee

Similarly,
given
$t\in \R^+$ and  $\tau\in \R$,
let $S^0 (t,\tau):\mathcal P_2(\mathbb L^2(
\mathcal Q
))\rightarrow\mathcal P_2(\mathbb L^2(
\mathcal Q
))$
be the map given by
$$
S^0(t,\tau)\mu
=p_{\ast}^0(\tau+t,\tau)
\mu ,
\quad \forall  \
\mu \in \mathcal P_2(\mathbb L^2(
\mathcal Q
)).
$$

Then $S^\varepsilon(t,\tau)$,
$t\ge \tau$,  is a continuous non-autonomous dynamical
system in  $\mathcal P_2(\mathbb L^2(\calo))$
as stated below.

\begin{lem}\label{dy} Suppose \eqref{cs1}-\eqref{La}   hold.
Then
 $S^\varepsilon(t,\tau)$,
$t\ge \tau$,
 is a continuous non-autonomous dynamical
 system in  $\mathcal P_2(\mathbb L^2(\calo))$ generated
by \eqref{aeu3};
more precisely,
   $S^\varepsilon(t,\tau):
 \mathcal P_2(\mathbb L^2(\calo))\rightarrow \mathcal P_2(\mathbb L^2(\calo))$
 satisfies the following conditions

(a) $S^\varepsilon(0,\tau)=I_{\mathcal P_2(\mathbb L^2(\calo))}$, for all $\tau\in \R$;

(b) $S^\varepsilon(s+t,\tau)=S^\varepsilon(t,\tau+s)\circ S^\varepsilon(s,\tau)$, for any $\tau \in \R$ and $t,s\in \R^+$;

(c)  $S^\varepsilon(t,\tau): \mathcal P_2(\mathbb L^2(\calo))\rightarrow \mathcal P_2(\mathbb L^2(\calo))$
is continuous, for every $\tau\in \R$ and $t\in \R^+$.
\end{lem}
\begin{proof}
	Note that
	(a) follows from the
	definition
	of $S^\varepsilon$,
	and (b) follows
	the Markov property
	of the solutions
	of \eqref{aeu3}.

We now prove  $(c)$.
Suppose  $\mu_n\rightarrow \mu$ in $\mathcal P_2(\mathbb L^2(\calo))$.
We will show
  $
  S^\varepsilon(t,\tau)
\mu_n\rightarrow S^\varepsilon(t,\tau)\mu$ in $\mathcal P_2(\mathbb L^2(\calo))$
for every $\tau\in \R$ and $t\in \R^+$.
Let  $\varphi\in C_b(\mathbb L^2(\calo))$.
By   Lemma \ref{markovp1}
we have
 $p^\varepsilon(\tau +t, \tau)
 \varphi\in C_b(\mathbb L^2(\calo))$
 for all
  $\tau\in \R$ and $t\in \R^+$,
  and hence
\begin{align}\label{pw}
\begin{split}
&\mathop {\lim }\limits_{n \to \infty } \left(\varphi,  {S^\varepsilon   \left( t,\tau \right)\mu _n  } \right)=
\mathop {\lim }\limits_{n \to \infty } \left(\varphi,  {p_ *^\varepsilon   \left( \tau+t,\tau \right)\mu _n   } \right)\\
 &= \mathop {\lim }\limits_{n \to \infty } \left( {   p^\varepsilon(\tau+t,\tau)  \varphi },
 \mu_n   \right)
 = \left( {p^\varepsilon(\tau+t,\tau)  \varphi }, \mu  \right)\\
  &= \left( \varphi,
   {p_ *^\varepsilon \left( \tau+t,\tau \right)\mu   } \right)
   = \left(\varphi,
    {S^\varepsilon \left( t,\tau \right)\mu   } \right),
\end{split}
\end{align}
 as desired.
\end{proof}

Next, we establish   the existence, uniqueness and periodicity  of  $\mathcal D_1$-pullback measure attractors
of  \eqref{aeu3} in  $\mathcal P_2 (\mathbb L^2(\calo))$.

\begin{thm}\label{Tie1}
If  \eqref{cs1}-\eqref{La} and \eqref{acg}  hold, then $S^\varepsilon$
 associated with problem \eqref{aeu3}
problem \eqref{aeu3} has a unique  $\mathcal D_1$-pullback measure attractor $\mathcal A_\varepsilon
=\left \{
\mathcal A_\varepsilon (\tau):
\tau\in \R
\right \} \in {\mathcal D_1 }$
in $\mathcal P_2(\mathbb L^2(\calo))$,
which is       given by, for each $\tau\in \R$,
\begin{align*}
\begin{split}
\mathcal A_\varepsilon\left( \tau \right) &= \omega(K,\tau)
 =\{\psi(0,\tau):\,\psi\,\text{is a}\,\,  \mathcal D_1 \text{-complete orbit of}\,\,S^\varepsilon\}\\
  &
 =\{\xi (\tau):\,\xi\,\text{is a}\,\,  \mathcal D_1 \text{-complete solution of}\,\,S^\varepsilon\},
\end{split}
\end{align*}
where $K=\{K(\tau):\,\tau\in \R\}$ is a  $\mathcal D_1$-pullback absorbing set of $S^\varepsilon$.
 Moreover, if $G_1$ and $G_2$ are $\rho$-periodic in $t$, then $S^\varepsilon$ associated with
\eqref{aeu3} has a unique $\rho$-periodic $\mathcal D_1$-pullback measure attractor
 $\mathcal A_\varepsilon$ in $\mathcal P_2(\mathbb L^2(\mathcal O))$.
\end{thm}

The next result is concerned
with the existence,
uniqueness and periodicity of
$\mathcal D_0$-pullback
 measure attractors for
$S^0$
associated with problem
(\ref{aeu4}), which is
 analogous to Theorem \ref{Tie1}.

\begin{thm}\label{Tie2}
 If  \eqref{cs1}-\eqref{La} and \eqref{acg} hold, then $S^0$ associated with
 problem \eqref{aeu4} has a unique $\mathcal D_0$-pullback
 measure attractor
 $\mathcal A_0
 =\left \{
 \mathcal A_0 (\tau):
 \tau\in \R
 \right \} \in {\mathcal D_0 }$
 in $\mathcal P_2(\mathbb L^2(\mathcal {Q}))$,
 which is       given by, for each $\tau\in \R$,
 \begin{align*}
 	\begin{split}
 		\mathcal A_0\left( \tau \right) &= \omega(K_0,\tau)
 	 =\{\psi(0,\tau):\,\psi\,\text{is a}\,\,  \mathcal D_0 \text{-complete orbit of}\,\,S^0\}\\
 		&
 		=\{\xi (\tau):\,\xi\,\text{is a}\,\,  \mathcal D_0 \text{-complete solution of}\,\,S^0\},
 	\end{split}
 \end{align*}
 where $K_0=\{K_0(\tau):\,\tau\in \R\}$ is  a  $\mathcal D_0$-pullback absorbing set of $S^0$.
 Moreover, if $G_1$ and $G_2$ are $\rho$-periodic in $t$, then $S^0$ associated with
\eqref{aeu4} has a unique $\rho$-periodic $\mathcal D_0$-pullback measure attractor
 $\mathcal A_0$ in $\mathcal P_2(\mathbb L^2(\mathcal Q))$.
\end{thm}

Finally,   we state the limiting behavior of $\mathcal D_1$-pullback measure attractors of \eqref{aeu3}
as $\eps \to 0$.

\begin{thm}\label{upsemi}
Assume that \eqref{cs1}-\eqref{La}, \eqref{acg} and \eqref{cg}-\eqref{cs}    hold.
Then the $\mathcal D_1$-pullback measure  attractors $\mathcal{A}_\varepsilon$ are
 upper semicontinuous at $\varepsilon=0$;
 more precisely,  for every $\tau\in \R$,
\[
\mathop {\lim }\limits_{\varepsilon  \to 0} \text{d}_{\mathcal P_2(\mathbb L^2 \left( {\mathcal O}  \right))}
 \left( {\mathcal{A}_\varepsilon(\tau) , \ \
 \mathcal{A}_0(\tau)
 \circ {\mathcal I}^{-1}
 }\right) = 0,
\]
where  $\mathcal I:
\mathbb L^2
(\mathcal {Q})
\to
\mathbb L^2(\mathcal {O})$ is the
 operator
  given by:
 for every
 $\varphi=(\varphi_1,\varphi_2)
 \in \mathbb L^2(\mathcal {Q}) $,
 $$
 \mathcal I \varphi  (y)
 = \varphi (y^*),
 \quad
 \forall \  y=(y^*, y_{n+1})
 \in \mathcal{O}.
 $$

\end{thm}

\section{Uniform Estimates}
\setcounter{equation}{0}

This section is devoted to
the  uniform estimates of the solutions of problem
\eqref{aeu3} and \eqref{aeu4} which are needed to show the
existence and limiting behavior of pullback  measure attractors.
In the sequel, we use
$\mathcal L(\xi)$ to denote the distribution law of a random variable $\xi$.

\begin{lem}\label{Tsb2}
Suppose \eqref{cs1}-\eqref{La} and \eqref{acg}  hold.
Then for every $\tau\in \R$ and $D_1=\{D_1(t):t\in \R\}\in \mathcal D_1$, there exists   $T=T(\tau,D_1)>0$,
independent of $\varepsilon$, such that for  all
$t\geq T$ and $0<\varepsilon<\varepsilon_0$, the solution $(u^\varepsilon,v^\varepsilon)$ of \eqref{aeu3} satisfies
\begin{align*}
\begin{split}
&\E\left( {\left\| {(u^\varepsilon\left( \tau,\tau-t,\xi^\varepsilon \right),
v^\varepsilon\left( \tau,\tau-t,\xi^\varepsilon \right))} \right\|^2_{\mathbb L^2 \left( \calo \right)} } \right)\\
&\le M_1+M_1\int_{-\infty}^\tau e^{ - \zeta \left( {\tau - s} \right)}
  \left (
  \left\|  {G_1(s,\cdot)} \right\|_{L^\infty( { {\widetilde {\mathcal O}} } )}^2
  +
  \left\|  {G_2(s,\cdot)} \right\|^2_{L^\infty( { {\widetilde {\mathcal O}} } )}
  \right ) ds
 \end{split}
 \end{align*}
 and
 \begin{align}\label{ee}
\begin{split}
& \int_{\tau-t}^\tau {e^{ - \zeta \left( {\tau - s} \right)}
  \E\left( {\left\| {u^\varepsilon\left( s,\tau- t,\xi^\varepsilon \right)} \right\|_{H_\varepsilon^1
  \left( \calo \right)}^2+\left\| {v^\varepsilon\left( s,\tau- t,\xi^\varepsilon \right)} \right\|_{L^2
  \left( \calo \right)}^2} \right)} ds\\
  &
  \le M_1+M_1\int_{-\infty}^\tau e^{ - \zeta \left( {\tau - s} \right)}
  \left (
  \left\|  {G_1(s,\cdot)} \right\|_{L^\infty( { {\widetilde {\mathcal O}} } )}^2
  +
  \left\|  {G_2(s,\cdot)} \right\|^2_{L^\infty( { {\widetilde {\mathcal O}} } )}
  \right ) ds ,
 \end{split}
 \end{align}
 where
  $\xi^\varepsilon  \in
 L^2_{\mathcal{F}_{\tau-t}}(\Omega, \mathbb  L^2(\mathcal {O}))$
 with
 $\mathcal L(\xi^\varepsilon)\in D_1(\tau-t)$, and
  $M_1>0$ is a   constant  independent of $\tau$, $\varepsilon$  and $D_1$.
\end{lem}

\begin{proof}
 By  \eqref{aeu3} and Ito's formula, we get for $t\geq \tau$,
 \begin{align*}
\begin{split}
&
\E\left(
 \beta\left\| {u^\varepsilon\left( t \right)} \right\|_{H_g \left( \calo \right)}^2+
 \alpha\left\| {v^\varepsilon\left( t \right)} \right\|_{H_g \left( \calo \right)}^2
 \right )
   + 2\beta  \E\left(
   \int_\tau^t {a_\varepsilon
  \left( {u^\varepsilon\left( s \right),u^\varepsilon\left( s \right)}
  \right)} ds
  \right )\\
 &\quad + 2\beta\lambda  \E\left({\int_\tau^t
 \left\|
   {u^\varepsilon\left( s \right)} \right\|
   _{H_g \left( \calo \right)}^2 ds}
   \right ) + 2\alpha\gamma \E\left({\int_\tau^t
 \left\|
   {v^\varepsilon\left( s \right)} \right\|
   _{H_g \left( \calo \right)}^2 ds}
   \right ) \\
 & =
 \E\left(
 \beta\left\| \xi^\varepsilon_1  \right\|_{H_g \left( \calo  \right)}^2
 +\alpha\left\| \xi^\varepsilon_2  \right\|_{H_g \left( \calo  \right)}^2
 \right )
   + 2\beta
   \E\left(
   \int_\tau^t {\left( {f_\varepsilon
   \left( {\cdot,u^\varepsilon\left( s \right)} \right),
   u^\varepsilon\left( s \right)} \right)} _{H_g \left( \calo  \right)} ds
   \right )\\
  &\quad + 2\beta  \E\left(
   \int_\tau^t {\left( {G_{1,\varepsilon}(s,\cdot)   ,
   u^\varepsilon\left( s \right)} \right)} _{H_g \left( \calo \right)} ds
   \right )+ 2\alpha  \E\left(
   \int_\tau^t {\left( {G_{2,\varepsilon}(s,\cdot)   ,
   v^\varepsilon\left( s \right)} \right)} _{H_g \left( \calo \right)} ds
   \right )\\
  &\quad+ \beta\sum\limits_{k = 1}^\infty
   \int_\tau^t
    \E\left(
     {\left\| {\sigma _{1,k,\varepsilon }
   + \varpi _{1,k} \left( {u^\varepsilon\left( s \right)} \right)}
    \right\|_{H_g \left( \calo  \right)}^2 }
    \right ) ds+ \alpha\sum\limits_{k = 1}^\infty
   \int_\tau^t
    \E\left(
     {\left\| {\sigma _{2,k,\varepsilon }
   + \varpi _{2,k} {v^\varepsilon\left( s \right)}}
    \right\|_{H_g \left( \calo  \right)}^2 }
    \right ) ds.
  \end{split}
 \end{align*}
 Consequently, we have for $t> \tau$,
 \begin{align}\label{sb1}
\begin{split}
&
{\frac d{dt}}
\E\left(
 \beta\left\| {u^\varepsilon\left( t \right)} \right\|_{H_g \left( \calo \right)}^2
 +\alpha \left\| {v^\varepsilon\left( t \right)} \right\|_{H_g \left( \calo \right)}^2
 \right )
   + 2\beta  \E\left(  {a_\varepsilon
  \left( {u^\varepsilon\left( t \right),u^\varepsilon\left( t \right)}
  \right)}
  \right )\\
 &\quad + 2\beta\lambda
  \E\left(
  {\left\|
   {u^\varepsilon\left( t \right)} \right\|
   _{H_g \left( \calo \right)}^2 }
   \right ) + 2\alpha\gamma
  \E\left(
  {\left\|
   {v^\varepsilon\left( t \right)} \right\|
   _{H_g \left( \calo \right)}^2 }
   \right )\\
 & =   2\beta
   \E\left(
    {\left( {f_\varepsilon
   \left( {\cdot ,u^\varepsilon\left(t \right)} \right),
   u^\varepsilon\left( t \right)} \right)} _{H_g \left( \calo  \right)}
   \right )\\
  &\quad + 2\beta
   \E\left(  {\left( {G_{1,\varepsilon} (t,\cdot)  ,
   u^\varepsilon\left( t \right)} \right)} _{H_g \left( \calo \right)}
   \right )+ 2\alpha
   \E\left(  {\left( {G_{2,\varepsilon} (t,\cdot)  ,
   u^\varepsilon\left( t \right)} \right)} _{H_g \left( \calo \right)}
   \right )\\
  &\quad+\beta \sum\limits_{k = 1}^\infty
    \E\left(
     {\left\| {\sigma _{1,k,\varepsilon }
   + \varpi _{1,k} \left( {u^\varepsilon\left( t\right)} \right)}
    \right\|_{H_g \left( \calo  \right)}^2 }
    \right ) + \alpha\sum\limits_{k = 1}^\infty
    \E\left(
     {\left\| {\sigma _{2,k,\varepsilon }
   + \varpi _{2,k} {v^\varepsilon\left( t\right)}}
    \right\|_{H_g \left( \calo  \right)}^2 }
    \right )  .
  \end{split}
 \end{align}

For the first   term on the right-hand side
of \eqref{sb1}, by  \eqref{a1} and \eqref{c3.1}  we obtain
\begin{align}\label{sb2}
\begin{split}
 &2\beta
  {\left( {f_\varepsilon  \left( {\cdot,u^\varepsilon\left( t
   \right)} \right),u^\varepsilon\left( s \right)}
  \right)_{H_g \left( \calo \right)} }
    = 2\beta  {\int_{\calo} {gf_\varepsilon
    \left( {y,u^\varepsilon\left( t \right)} \right)}
     u^\varepsilon\left( t \right)dy }  \\
 & \le  -2\beta \alpha _1 \underline g
  {\int_{\calo} {\left| {u^\varepsilon\left( t \right)} \right|^p dy}  }
  +2\beta \overline g |\calo| \left\| {\psi _1 }
  \right\|_{L^\infty  \left( {\widetilde {\calo}} \right)} ,
  \end{split}
 \end{align}
where $|\calo|$ is the Lebesgue measure of $\calo$.
 For the second and third  term on the right-hand side of \eqref{sb1}, we have
\begin{align}\label{sb3}
\begin{split}
 &2\beta
  {\left( {G_{1,\varepsilon}(t,\cdot)    ,u^\varepsilon\left( t
   \right)} \right)} _{H_g \left( \calo \right)}+2\alpha
  {\left( {G_{2,\varepsilon } (t,\cdot)  ,v^\varepsilon\left( t
   \right)} \right)} _{H_g \left( \calo \right)}\\
 & \le {\lambda}
  ({\beta\left\| {u^\varepsilon\left(t \right)}
  \right\|_{H_g \left( \calo \right)}^2+\alpha\left\| {v^\varepsilon\left(t \right)}
  \right\|_{H_g \left( \calo \right)}^2 })
  + \frac{1}{{\lambda }}\overline g |\calo|(\beta \left\|
   {G_1(t,\cdot) } \right\|_{L^\infty  \left( {\widetilde \calo} \right)}^2+\alpha\left\|
   {G_2(t,\cdot) } \right\|_{L^\infty  \left( {\widetilde \calo} \right)}^2 ) .
  \end{split}
 \end{align}
For the last two terms on the right-hand
side of \eqref{sb1}, we have from \eqref{cv1}
\begin{align}\label{sb4}
\begin{split}
 &\beta\sum\limits_{k = 1}^\infty
   {  {\left\| {\sigma _{1,k,\varepsilon }
 + \varpi _{1,k} \left( {u^\varepsilon\left( t \right)} \right)}
  \right\|_{H_g \left( \calo \right)}^2 } } +\alpha\sum\limits_{k = 1}^\infty
   {  {\left\| {\sigma _{2,k,\varepsilon }
 + \varpi _{2,k} {v^\varepsilon\left( t \right)}}
  \right\|_{H_g \left( \calo \right)}^2 } }  \\
  &\le 2
  \overline g |\calo|( \beta\sum\limits_{k = 1}^\infty  {\left\|
  {\sigma _{1,k} } \right\|_{L^\infty
   \left( {\widetilde \calo} \right)}^2 }+\alpha\sum\limits_{k = 1}^\infty  {\left\|
  {\sigma _{2,k} } \right\|_{L^\infty
   \left( {\widetilde \calo} \right)}^2 })\\
  &\quad + 4\overline g |\calo|
\beta\sum\limits_{k = 1}^\infty
    {\beta _{k}^2 }
     + \sum\limits_{k = 1}^\infty  {(4\gamma _{k}^2+2\varpi _{2,k}^2)
    {\left( {\beta\left\| {u^\varepsilon\left( t \right)}
     \right\|_{H_g \left( \calo \right)}^2+\alpha\left\| {v^\varepsilon\left( t \right)}
     \right\|_{H_g \left( \calo \right)}^2 } \right)} } .
 \end{split}
 \end{align}
 It follows from   \eqref{sb1}-\eqref{sb4}
 that for all $t> \tau$,
\begin{align}\label{sb5}
\begin{split}
&
{\frac {d}{dt}}
\E\left( {\beta\left\| {u^\varepsilon\left( t \right)}
\right\|_{H_g \left( \calo \right)}^2+\alpha\left\| {v^\varepsilon\left( t \right)}
\right\|_{H_g \left( \calo \right)}^2 } \right)\\
&\quad+( \lambda-4 \sum\limits_{k = 1}^\infty  (\gamma _{k}^2+\varpi _{2,k}^2))\E
\left( {\beta\left\| {u^\varepsilon\left( t \right)}
\right\|_{H_g \left( \calo \right)}^2+\alpha\left\| {v^\varepsilon\left( t \right)}
\right\|_{H_g \left( \calo \right)}^2 } \right)\\
&\quad +
2\beta{\E\left( {a_\varepsilon  \left( {u^\varepsilon\left( t \right),
u^\varepsilon\left( t \right)} \right)} \right)}+
2\beta\alpha _1 \underline g
 {\E\left(\left\| {u^\varepsilon\left( t
  \right)} \right\|_{L^p \left( \calo\right)}^p \right)}+2\varpi_2\left\| {v^\varepsilon\left( t \right)}
\right\|_{H_g \left( \calo \right)}^2\\
  &\le \frac{1}{{\lambda }}\overline g |\calo|(\beta \left\|
   {G_1(t,\cdot) } \right\|_{L^\infty  \left( {\widetilde \calo} \right)}^2+\alpha\left\|
   {G_2(t,\cdot) } \right\|_{L^\infty  \left( {\widetilde \calo} \right)}^2 )+ c_1,
 \end{split}
 \end{align}
where $\varpi_2=\sum\limits_{k = 1}^\infty \varpi _{2,k}^2$.

Multiplying \eqref{sb5} by  $e^{\zeta t}$ and then integrating the resulting
inequality on $(\tau-t,\tau)$ with $t\in \R^+$, we get
\begin{align}\label{sb6}
\begin{split}
 &\E\left( {\left\| {(\beta u^\varepsilon\left(\tau,\tau- t,\xi^\varepsilon \right)
 ,\alpha v^\varepsilon\left(\tau,\tau- t,\xi^\varepsilon \right))}
 \right\|_{\mathbb H_g \left( \calo \right)}^2 } \right) \\
 &\quad+ 2\beta\int_{\tau-t}^\tau
 {e^{ - \zeta \left( {\tau - s} \right)} \E\left( {a_\varepsilon
 \left( {u^\varepsilon\left( s, \tau- t,\xi^\varepsilon \right),
 u^\varepsilon\left( s, \tau- t,\xi^\varepsilon \right)} \right)} \right)} ds \\
 &\quad + 2\beta\alpha _1
 \underline g \int_{\tau-t}^\tau {e^{ - \zeta \left( {\tau - s} \right)}
  \E\left( {\left\| {u^\varepsilon\left( s,\tau- t,\xi^\varepsilon \right)} \right\|_{L^p
  \left( \calo \right)}^p} \right)} ds\\
 &\quad +2\varpi_2
 \underline g \int_{\tau-t}^\tau {e^{ - \zeta \left( {\tau - s} \right)}
  \E\left( {\left\| {v^\varepsilon\left( s,\tau- t,\xi^\varepsilon \right)} \right\|_{L^2
  \left( \calo \right)}^2} \right)} ds  \\
  & \le \E\left (
  \left\| \xi^\varepsilon  \right\|_{\mathbb H_g
 \left( \calo \right)}^2 \right )
 e^{ - \zeta t} +\frac{1}{\lambda}\overline g |\calo|\int_{\tau-t}^\tau {e^{ - \zeta \left( {\tau - s} \right)}
  (\beta\left\| {G_1(s,\cdot)  } \right\|^2_{L^\infty
   \left( {\widetilde \calo} \right)}+\alpha\left\| {G_2(s,\cdot)  } \right\|^2_{L^\infty
   \left( {\widetilde \calo} \right)})} ds + \frac{{1 }}{\zeta }c_1\\
  & \le \E\left (
  \left\| \xi^\varepsilon  \right\|_{\mathbb H_g
 \left( \calo \right)}^2 \right )
 e^{ - \zeta t} +\frac{1}{\lambda}\overline g |\calo|\int_{-\infty}^\tau {e^{ - \zeta \left( {\tau - s} \right)}
  (\beta\left\| {G_1(s,\cdot)  } \right\|^2_{L^\infty
   \left( {\widetilde \calo} \right)}+\alpha\left\| {G_2(s,\cdot)  } \right\|^2_{L^\infty
   \left( {\widetilde \calo} \right)})} ds + \frac{{1 }}{\zeta }c_1.
 \end{split}
 \end{align}
Since   $\mathcal L(\xi^\varepsilon)  \in D_1\left( {\tau  - t } \right)$ we have
$$
e^{ -\zeta t} \E\left (
\left\| \xi^\varepsilon  \right\|_{\mathbb H_g
	\left( \calo \right)}^2 \right )
\le \overline g
 e^{ -\zeta t} \E\left (
\left\| \xi^\varepsilon  \right\|_{\mathbb L^2
	\left( \calo \right)}^2 \right )
\le
\overline g
e^{ - \zeta t} \left\| {D_1 \left( {\tau  - t } \right)} \right\|_{
 	\mathbb L^2  \left( \mathcal{O} \right)}^2
 \to 0,\quad \text{ as }
 t\to \infty,
$$
and hence
   there exists $T = T (\tau,  D_1) > 0$ such that for all $t\geq T$,
\[
e^{ -\zeta t} \E\left (
  \left\| \xi^\varepsilon  \right\|_{\mathbb H_g
 \left( \calo \right)}^2 \right )
   \le \frac{{c_1 }}{\zeta},
\]
which along with
  \eqref{sb6} and \eqref{4.1}
  concludes the proof.

\end{proof}

By Lemma \ref{Tsb2}, we have the following uniform estimates.

\begin{lem}\label{aTsb2}
Assume  \eqref{cs1}-\eqref{La} and \eqref{acg} hold. Then
for every $\tau\in \R$ and $D_1=\{D_1(t):t\in \R\}\in \mathcal D_1$, there exists   $T=T(\tau,D_1)>1$,
independent of $\varepsilon$, such that for  all
$t\geq T$ and $0<\varepsilon<\varepsilon_0$,
\begin{align*}
\begin{split}
&\int_{\tau  - 1}^\tau  {\E\left( {\left\| {u^\varepsilon  \left( s,\tau-t,\xi^\varepsilon
 \right)}
 \right\|_{H^1_\varepsilon \left( \calo \right)}^2+\left\| {v^\varepsilon\left( s,\tau- t,\xi^\varepsilon \right)} \right\|_{L^2
  \left( \calo \right)}^2 } \right)} ds \\
  &\quad \le M_2
 +M_2
 \int_{-\infty}^\tau e^{ -\zeta \left( {\tau - s} \right)}
\left (
\left\|  {G_1(s,\cdot)} \right\|_{L^\infty( { {\widetilde {\mathcal O}} } )}^2
+
\left\|  {G_2(s,\cdot)} \right\|^2_{L^\infty( { {\widetilde {\mathcal O}} } )}
\right ) ds
 ,
 \end{split}
\end{align*}
 where
 $\xi^\varepsilon  \in
 L^2_{\mathcal{F}_{\tau-t}}(\Omega, \mathbb  L^2(\mathcal {O}))$
 with
$\mathcal L(\xi^\varepsilon)\in D_1(\tau-t)$, and
  $M_2>0$ is a   constant  independent of $\tau$, $\varepsilon$  and $D_1$.
\end{lem}

Recall   the following inequality
from \cite[Lemma 3.3]{li2017limiting} which will be used
to derive    the
uniform estimates of solutions in $H_\varepsilon^1(\mathcal O).$

\begin{lem}\label{Tsb3}
If  (\ref{cf1})-(\ref{cf4})  hold, then for $u\in D(A_\varepsilon)$,
  $$
  (f_\varepsilon(\cdot,u), A_\varepsilon u)_{H_g(\mathcal{O})} \leq
 M \left (
 a_\varepsilon(u,u)+\|\psi_3\|_{L^\infty(
  { {\widetilde {\mathcal O}} } )}^2
 \right ),
 $$
 where $M$ is a positive constant   independent of $\varepsilon$.
\end{lem}

\begin{lem}\label{Tsb4}
Suppose \eqref{cs1}-\eqref{La} and \eqref{acg} hold.
  Then for every $\tau\in \R$ and $D_1=\{D_1(t):t\in \R\}\in \mathcal D_1$, there exists   $T=T(\tau,D_1)>1$,
independent of $\varepsilon$, such that for  all
$t\geq T$ and $0<\varepsilon<\varepsilon_0$,
 \begin{align}\label{aL5.1}
 \begin{split}
&\E(\left\| {u^\varepsilon\left( {\tau,\tau-t,\xi^\varepsilon} \right)} \right\|_{H^1_\varepsilon(\mathcal O)  }^2)
\le M_3+M_3\int_{-\infty}^\tau e^{ - \zeta \left( {\tau - s} \right)}
  \left (
\left\|  {G_1(s,\cdot)} \right\|_{L^\infty( { {\widetilde {\mathcal O}} } )}^2
+
\left\|  {G_2(s,\cdot)} \right\|^2_{L^\infty( { {\widetilde {\mathcal O}} } )}
\right )ds,
  \end{split}
 \end{align}
where
$\xi^\varepsilon  \in
L^2_{\mathcal{F}_{\tau-t}}(\Omega,  \mathbb L^2(\mathcal {O}))$
with
  $\mathcal L(\xi^\varepsilon)\in D_1(\tau-t)$,
  and    $M_3>0$ is a   constant  independent of $\tau$, $D_1$
  and   $\varepsilon$.
\end{lem}

\begin{proof}
By Ito's formula and \eqref{aeu3}, we
get for    $\tau\in \R$, $t>1$ and  $\varsigma\in (\tau-1,\tau)$,
\begin{align}\label{atu1}
\begin{split}
 &a_\varepsilon  \left( {u^\varepsilon\left( \tau,\tau-t,\xi^\varepsilon \right),
 u^\varepsilon\left( \tau,\tau-t,\xi^\varepsilon \right)} \right) +
  2\int_\varsigma^\tau {\left\| {A_\varepsilon
  u^\varepsilon\left( s,\tau-t,\xi^\varepsilon \right)} \right\|_{H_g \left(\calo \right)}^2 ds}\\
   &\quad  + 2\lambda \int_\varsigma^\tau  {a_\varepsilon
   \left( {u^\varepsilon\left( s,\tau-t,\xi^\varepsilon \right),
   u^\varepsilon\left( s,\tau-t,\xi^\varepsilon \right)}
     \right)  ds} \\
  &\quad   + 2\alpha \int_\varsigma^\tau  {(v^\varepsilon\left( s,\tau-t,\xi^\varepsilon \right),A_\varepsilon
  u^\varepsilon\left( s,\tau-t,\xi^\varepsilon \right))_{H_g \left( \calo \right)} ds} \\
 & = a_\varepsilon  \left( {u^\varepsilon\left( \varsigma,\tau-t,\xi^\varepsilon \right),
 u^\varepsilon\left(\varsigma,\tau-t,\xi^\varepsilon \right)} \right)\\
 &\quad +2\int_\varsigma^\tau  {\left( {f_\varepsilon  \left( {\cdot
 ,u^\varepsilon\left( s,\tau-t,\xi^\varepsilon \right)} \right),
 A_\varepsilon  u^\varepsilon\left( s,\tau-t,\xi^\varepsilon \right)}
  \right)_{H_g \left( \calo \right)} ds}  \\
  &\quad+ 2\int_\varsigma^\tau  {\left( {G_{1,\varepsilon}(s,\cdot) ,A_\varepsilon
    u^\varepsilon\left( s,\tau-t,\xi^\varepsilon \right)} \right)_{H_g \left( \calo \right)} ds}  \\
 &\quad + \sum\limits_{k = 1}^\infty  {\int_\varsigma^\tau  {a_\varepsilon
   \left( {\sigma _{1,k,\varepsilon }
  +\varpi_{1,k} \left( {u^\varepsilon\left( s,\tau-t,\xi^\varepsilon \right)}
   \right),\sigma _{1,k,\varepsilon }
   + \varpi _{1,k} \left( {u^\varepsilon\left( s,\tau-t,\xi^\varepsilon \right)} \right)} \right)} } ds \\
  &\quad+ 2\sum\limits_{k = 1}^\infty
   {\int_\varsigma^\tau  {\left( {\sigma _{1,k,\varepsilon }
    + \varpi _{1,k} \left( {u^\varepsilon\left( s,\tau-t,\xi^\varepsilon \right)} \right),A_\varepsilon
   u^\varepsilon\left( s,\tau-t,\xi^\varepsilon \right)} \right)_{H_g \left( \calo \right)}} } dW_k \left( s \right).
 \end{split}
\end{align}

For the fourth  term on the left-hand side of (\ref{atu1}) is bounded by
\begin{align}\label{aatu3}
\begin{split}
 &|2\alpha\int_\varsigma^t {\left( {v^\varepsilon\left( s,\tau-t,\xi^\varepsilon \right),
 A_\varepsilon  u^\varepsilon\left( s,\tau-t,\xi^\varepsilon \right)} \right)_{H_g \left( \calo \right)} ds} | \\
 & \le
  \int_\varsigma^t {\left\| {A_\varepsilon  u^\varepsilon\left( s,\tau-t,\xi^\varepsilon \right)} \right\|_{H_g \left( \calo\right)}^2 ds}
   +\alpha^2\int_\varsigma^t\|v^\varepsilon\left( s,\tau-t,\xi^\varepsilon \right)\|_{H_g \left( \calo\right)}^2ds.
 \end{split}
\end{align}
It follows from   Lemma {\ref{Tsb3}} that
\begin{align}\label{atu2}
\begin{split}
 &2\int_\varsigma^\tau {\left( {f_\varepsilon  \left( {\cdot,u^\varepsilon\left( s,\tau-t,\xi^\varepsilon \right)} \right),A_\varepsilon
  u^\varepsilon\left( s,\tau-t,\xi^\varepsilon \right)} \right)_{H_g \left( \calo \right)} ds}  \\
  &\le   2c_1\int_\varsigma^\tau {a_\varepsilon  \left( {u^\varepsilon\left( s,\tau-t,\xi^\varepsilon \right)
  ,u^\varepsilon\left( {{s,\tau-t,\xi^\varepsilon}} \right)} \right)ds}  +
  2c_1\left\| {\psi _3 } \right\|_{L^\infty  \left( {\widetilde \calo} \right)}^2(\tau-\varsigma) .
 \end{split}
\end{align}
For  the third term on the right-hand side of (\ref{atu1})
we have
\begin{align}\label{atu3}
\begin{split}
 &2\int_\varsigma^\tau {\left( {G_{1,\varepsilon}(s,\cdot) ,A_\varepsilon
  u^\varepsilon\left( s,\tau-t,\xi^\varepsilon \right)} \right)_{H_g \left( \calo \right)} ds}  \\
 & \le \int_\varsigma^\tau {\left\| {A_\varepsilon
  u^\varepsilon\left( s,\tau-t,\xi^\varepsilon \right)} \right\|_{H_g \left( \calo\right)}^2 ds}
   +\overline g |\calo| \int_\varsigma^\tau\left\| {G_1(s,\cdot)}
    \right\|_{L^\infty  \left( {\widetilde \calo} \right)}^2 ds.
 \end{split}
\end{align}
By \eqref{cv3},
after simple calculations,
we   get
\begin{align}\label{atu4}
\begin{split}
 &\sum\limits_{k = 1}^\infty  {\int_\varsigma^\tau {a_\varepsilon
  \left( {\sigma _{1,k,\varepsilon }
 + \varpi _k \left( {u^\varepsilon\left( s,\tau-t,\xi^\varepsilon \right)} \right),
 \sigma _{1,k,\varepsilon }
 + \varpi_{1,k} \left( {u^\varepsilon\left( s,\tau-t,\xi^\varepsilon \right)} \right)} \right)ds} }  \\
 & \le 2 \sum\limits_{k = 1}^\infty
   {\int_\varsigma^\tau {\left( {a_\varepsilon
  \left(
   {\varpi_{1,k} \left( {u^\varepsilon\left( s,\tau-t,\xi^\varepsilon \right)} \right)
   ,\varpi _{1,k} \left( {u^\varepsilon\left( s,\tau-t,\xi^\varepsilon \right)} \right)} \right)
   +  a_\varepsilon  \left( {\sigma _{1,k,\varepsilon }
    ,\sigma _{1,k,\varepsilon } } \right)}
    \right )
    ds} }  \\
  &\le 2 \sum\limits_{k = 1}^\infty
   {  L_k ^2 } \int_\varsigma^\tau {a_\varepsilon
    \left( {u^\varepsilon\left( s,\tau-t,\xi^\varepsilon \right),
    u^\varepsilon\left( s,\tau-t,\xi^\varepsilon \right)} \right)ds}  +
    \overline g |\calo|
    \sum\limits_{k = 1}^\infty
     {\left\| {\sigma_{k}^{*}} \right\|_{L^\infty  \left( {\widetilde \calo} \right)}^2 }(\tau-\varsigma).
 \end{split}
\end{align}

By (\ref{atu1})-(\ref{atu4})   we
obtain
\begin{align}\label{atu5}
\begin{split}
 &\E\left( {a_\varepsilon  \left( {u^\varepsilon\left( \tau,\tau-t,\xi^\varepsilon \right),
 u^\varepsilon\left( \tau,\tau-t,\xi^\varepsilon \right)} \right)} \right)\\
  &\le\E\left( {a_\varepsilon  \left( {u^\varepsilon\left( \varsigma \right),
  u^\varepsilon\left( \varsigma \right)} \right)} \right)
   +\alpha^2\int_\varsigma^t\E\left(\|v^\varepsilon\left( s,\tau-t,\xi^\varepsilon \right)\|_{H_g \left( \calo\right)}^2\right)ds\\
  &\quad + c_2 \int_\varsigma^\tau {\E\left( {a_\varepsilon
   \left( {u^\varepsilon\left( s,\tau-t,\xi^\varepsilon \right),
   u^\varepsilon\left( s,\tau-t,\xi^\varepsilon \right)} \right)} \right)ds}
   \\
 &\quad+\overline g |\calo| \int_\varsigma^\tau\left\| {G_1(s,\cdot)}
    \right\|_{L^\infty  \left( {\widetilde \calo} \right)}^2 ds  +  (\overline g |\calo|
    \sum\limits_{k = 1}^\infty
     {\left\| {\sigma_{k}^{*}} \right\|_{L^\infty  \left( {\widetilde \calo} \right)}^2 }+
     2c_1\left\| {\psi _3 } \right\|_{L^\infty  \left( {\widetilde \calo} \right)}^2)(\tau-\varsigma).
 \end{split}
\end{align}
Integrating the inequality
 with respect to $\varsigma$ over $(\tau-1, \tau)$, we find
\begin{align}\label{atu6}
\begin{split}
& \E\left( {a_\varepsilon  \left( {u^\varepsilon\left( \tau,\tau-t,\xi^\varepsilon \right),
 u^\varepsilon\left( \tau,\tau-t,\xi^\varepsilon \right)} \right)} \right)\\
 &\le \left( {c_2  + 1} \right)\int_{\tau - 1}^\tau {\E\left( {a_\varepsilon
  \left( {u^\varepsilon\left( s,\tau-t,\xi^\varepsilon \right),
  u^\varepsilon\left( s,\tau-t,\xi^\varepsilon \right)} \right)} \right)ds}
   +\alpha^2\int_{\tau-1}^\tau\E\left(\|v^\varepsilon\left( s,\tau-t,
   \xi^\varepsilon \right)\|_{H_g \left( \calo\right)}^2\right)ds\\
 &\quad   +\overline g |\calo| \int_{\tau-1}^\tau\left\| {G_1(s,\cdot)}
    \right\|_{L^\infty  \left( {\widetilde \calo} \right)}^2 ds  +  \overline g |\calo|
    \sum\limits_{k = 1}^\infty
     {\left\| {\sigma_{k}^{*}} \right\|_{L^\infty  \left( {\widetilde \calo} \right)}^2 }+
     2c_1\left\| {\psi _3 } \right\|_{L^\infty  \left( {\widetilde \calo} \right)}^2\\
 &\le \left( {c_2  + 1} \right)\int_{\tau - 1}^\tau {\E\left( {a_\varepsilon
  \left( {u^\varepsilon\left( s,\tau-t,\xi^\varepsilon \right),
  u^\varepsilon\left( s,\tau-t,\xi^\varepsilon \right)} \right)} \right)ds}
   +\alpha^2\int_{\tau-1}^\tau\E\left(\|v^\varepsilon\left( s,\tau-t,\xi^\varepsilon \right)\|_{H_g \left( \calo\right)}^2\right)ds\\
 &\quad   +\overline g |\calo|e^{\gamma\tau} \int_{\tau-1}^\tau e^{-\gamma(\tau-s)} \left\| {G_1(s,\cdot)}
    \right\|_{L^\infty  \left( {\widetilde \calo} \right)}^2 ds  +  \overline g |\calo|
    \sum\limits_{k = 1}^\infty
     {\left\| {\sigma_{k}^{*}} \right\|_{L^\infty  \left( {\widetilde \calo} \right)}^2 }+
     2c_1\left\| {\psi _3 } \right\|_{L^\infty  \left( {\widetilde \calo} \right)}^2.
 \end{split}
\end{align}
By \eqref{atu6}  and Lemma
  \ref{aTsb2}, we find that for every $\tau\in \R$
  and $D_1=\{D_1(t):t\in \R\}\in \mathcal D_1$, there exists   $T=T(\tau,D_1)>1$,
independent of $\varepsilon$, such that for  all
$t\geq T$, $0<\varepsilon<\varepsilon_0$
 and
$\xi^\varepsilon  \in L^2_{\mathcal{F}_{\tau-t}}(\Omega, \mathbb  L^2(\mathcal {O}))$
with
 $\mathcal  L(\xi^\varepsilon)\in D_1(\tau-t)$,
\begin{align*}\label{atu7}
\begin{split}
 & \E\left(
 a_\varepsilon  \left( u^\varepsilon(\tau,\tau-t,\xi^\varepsilon),u^\varepsilon(\tau,,\tau-t,\xi^\varepsilon) \right)
 \right ) \le c_3+c_3\int_{-\infty}^\tau e^{ - \zeta \left( {\tau - s} \right)}
 \left (
\left\|  {G_1(s,\cdot)} \right\|_{L^\infty( { {\widetilde {\mathcal O}} } )}^2
+
\left\|  {G_2(s,\cdot)} \right\|^2_{L^\infty( { {\widetilde {\mathcal O}} } )}
\right )ds,
 \end{split}
\end{align*}
which together with Lemma \ref{Tsb2} completes the proof.
\end{proof}

In order to establish the uniform asymptotic compactness of $\mathcal L(v^\varepsilon)$, we need to decompose
$v^\varepsilon$ as a sum of two functions: one is regular in the sense that it belongs to $H^1(\calo)$ and the
other converges to zero in mean square moment as $t\rightarrow\infty$.
 We split $v^\varepsilon$ as $v^\varepsilon=v^\varepsilon_1+v^\varepsilon_2$
where $v_1^\varepsilon$ is the solution of the initial value problem, for $t>\tau$,
\be \label{v1}
dv_1^\varepsilon  \left( t \right) + \gamma v_1^\varepsilon  \left( t \right)
= \sum\limits_{k = 1}^\infty  {\varpi _{2,k} v_1^\varepsilon  \left( t \right)dW_k \left( t \right)},\quad
v_1^\varepsilon(\tau)=\xi_2^\varepsilon(y),\quad y\in \calo,
\ee
and $v^\varepsilon_2$ is the solution of
\begin{align}\label{v21}
\begin{split}
 dv_2^\varepsilon  \left( t \right)& + \gamma v_2^\varepsilon  \left( t \right)dt - \beta u^\varepsilon  dt
 = G_{2,\varepsilon } \left(t, y \right)dt \\
 &+\sum\limits_{k = 1}^\infty  {\left( {\sigma _{2,k,\varepsilon } \left( y \right) +
 \varpi _{2,k} v_2^\varepsilon  \left( t \right)} \right)dW_k \left( t \right)} ,\quad y \in \calo,\quad t > \tau,
 \end{split}
 \end{align}
with initial condition
\be \label{v22}
v^\varepsilon_2(\tau)=0.
\ee
Under condition \eqref{La}, we easily have for every $\tau\in \R$ and $D_1=\{D_1(t):t\in \R\}\in \mathcal D_1$
\begin{align}\label{v1e}
\begin{split}
&\E\left( {\left\| {v_1^\varepsilon  \left(\tau, \tau-t,\xi_2^\varepsilon \right)} \right\|^2 } \right) \le
\E\left(\left\| {\xi _2 } \right\|^2\right) e^{ - \left( {2\gamma  -
\sum\limits_{k = 1}^\infty  {\varpi _{2,k}^2 } } \right)t}\\
&\le\left\| {D_1 } \right\|^2_{\mathcal P_2(\mathbb L^2(\calo))} e^{ - \left( {2\gamma  -
\sum\limits_{k = 1}^\infty  {\varpi _{2,k}^2 } } \right)t}  \to 0, \quad\text{as}\quad t\rightarrow\infty,
 \end{split}
 \end{align}
where
$\xi^\varepsilon=(\xi_1^\varepsilon,\xi^\varepsilon_2)  \in
L^2_{\mathcal{F}_{\tau-t}}(\Omega,   \mathbb L^2(\mathcal {O}))$
with  $\mathcal L(\xi^\varepsilon)\in D_1(\tau-t)$.

Next, we derive uniform estimates for $v^\varepsilon_2$ in ${H_\varepsilon ^1 \left( \calo \right)}$.

\begin{lem}\label{aTsb4}
Suppose \eqref{cs1}-\eqref{La} and \eqref{acg} hold.
  Then for every $\tau\in \R$ and $D_1=\{D_1(t):t\in \R\}\in \mathcal D_1$, there exists   $T=T(\tau,D_1)>1$,
independent of $\varepsilon$, such that for  all
$t\geq T$ and $0<\varepsilon<\varepsilon_0$, the solution $v^\varepsilon_2$ of \eqref{v21}-\eqref{v22} satisfies
 \begin{align*}
 \begin{split}
&\E(\left\| {v^\varepsilon_2\left( {\tau,\tau-t,0} \right)} \right\|_{H^1_\varepsilon(\mathcal O)  }^2)
\le M_4+M_4\int_{-\infty}^\tau e^{ - \zeta \left( {\tau - s} \right)}
 \left(
  \left\|  {G_1(s,\cdot)} \right\|_{L^\infty( { {\widetilde {\mathcal O}} } )}^2
  +
  \left\|  {G_2(s,\cdot)} \right\|^2_{H^{1,\infty}( { {\widetilde {\mathcal O}} } )}
  \right )ds,
  \end{split}
 \end{align*}
where $M_4>0$ is a   constant  independent of $\tau$, $D_1$
  and   $\varepsilon$.
\end{lem}

\begin{proof}
By Ito's formula and \eqref{v21}, we get for $\tau\in \R$ and $t>0$,
\begin{align}\label{tu1}
\begin{split}
 &a_\varepsilon  \left( {v^\varepsilon_2\left( \tau,\tau-t,0 \right),v^\varepsilon_2\left( \tau,\tau-t,0 \right)} \right)
    + 2\gamma \int_{\tau-t}^\tau {a_\varepsilon
     \left( {v^\varepsilon_2\left( s,\tau-t,0 \right),v^\varepsilon_2\left( s,\tau-t,0 \right)}
     \right)  ds}  \\
 & =2\beta\int_{\tau-t}^\tau {\left( {u^\varepsilon(s,\tau-t,\xi^\varepsilon) ,A_\varepsilon
    v^\varepsilon_2\left( s,\tau-t,0 \right)} \right)_{H_g \left( \calo \right)} ds}\\
   &\quad  + 2\int_{\tau-t}^\tau {\left( {G_{2,\varepsilon}(s,\cdot) ,A_\varepsilon
    v^\varepsilon_2\left( s,\tau-t,0 \right)} \right)_{H_g \left( \calo \right)} ds}   \\
 &\quad + \sum\limits_{k = 1}^\infty  {\int_{\tau-t}^\tau {a_\varepsilon
   \left( {\sigma _{2,k,\varepsilon }
  +\varpi_{2,k} {v^\varepsilon_2\left( s,\tau-t,0 \right)},\sigma _{2,k,\varepsilon }
   + \varpi _{2,k} {v^\varepsilon_2\left( s,\tau-t,0 \right)}} \right)} } ds \\
  &\quad+ 2\sum\limits_{k = 1}^\infty
   {\int_{\tau-t}^\tau {\left( {\sigma _{2,k,\varepsilon }
    + \varpi _{2,k} {v^\varepsilon_2\left( s,\tau-t,0 \right)},A_\varepsilon
   v^\varepsilon_2\left( s,\tau-t,0 \right)} \right)_{H_g \left( \calo \right)}} } dW_k \left( s \right).
 \end{split}
\end{align}

We first estimate the first term
on the right-hand side of  (\ref{tu1}). We have
\begin{align}\label{tu2}
\begin{split}
 &2\beta \int_{\tau-t}^\tau {\left( {u^\varepsilon  \left( s,\tau-t,\xi^\varepsilon  \right),
 A_\varepsilon  v^\varepsilon_2  \left( s,\tau-t,0 \right)} \right)} ds\\
&  = 2\beta \int_{\tau-t}^\tau {a_\varepsilon  \left( {u^\varepsilon  \left( s,\tau-t,\xi^\varepsilon  \right),
  v^\varepsilon_2  \left( s,\tau-t,0 \right)} \right)} ds \\
 & \le \frac{\gamma }{2}\int_{\tau-t}^\tau {a_\varepsilon  \left( {v^\varepsilon_2
 \left( s,\tau-t,0 \right),v^\varepsilon_2  \left( s,\tau-t,0 \right)} \right)} ds\\
&\quad + \frac{2}{\gamma }\beta ^2 \int_{\tau-t}^\tau {a_\varepsilon
  \left( {u^\varepsilon  \left( s,\tau-t,\xi^\varepsilon  \right),
  u^\varepsilon  \left( s,\tau-t,\xi^\varepsilon  \right)} \right)} ds.
 \end{split}
\end{align}
On the other hand, the second term on the right-hand side of (\ref{tu1}) is bounded by
\begin{align}\label{tu3}
\begin{split}
 &2 \int_{\tau-t}^\tau {\left( {G_{2,\varepsilon }(s,\cdot) ,A_\varepsilon  v^\varepsilon_2  \left( s,\tau-t,0 \right)} \right)} ds
 = 2 \int_{\tau-t}^\tau {a_\varepsilon  \left( {G_{2,\varepsilon} (s,\cdot),v^\varepsilon_2  \left( s,\tau-t,0 \right)} \right)} ds \\
 & \le \frac{\gamma }{2}\int_{\tau-t}^\tau {a_\varepsilon  \left( {v^\varepsilon_2  \left( s,\tau-t,0 \right),v^\varepsilon_2
   \left( s \right)} \right)} ds + \frac{2}{\gamma } \int_{\tau-t}^\tau {a_\varepsilon
    \left( {G_{2,\varepsilon }(s,\cdot) ,G_{2,\varepsilon }(s,\cdot) } \right)} ds.
 \end{split}
\end{align}
For the  third term on the right-hand side of (\ref{tu1}), by
the similar arguments as that in Lemma 3.3 in \cite{li2017limiting} we have
\begin{align}\label{tu4}
\begin{split}
 &\sum\limits_{k = 1}^\infty  {\int_{\tau-t}^\tau {a_\varepsilon
  \left( {\sigma _{2,k,\varepsilon }
 + \varpi _{2,k} {v^\varepsilon_2\left( s,\tau-t,0 \right)},
 \sigma _{2,k,\varepsilon }
 + \varpi_{2,k} {v^\varepsilon_2\left( s,\tau-t,0 \right)}} \right)ds} }  \\
 & \le 2 \sum\limits_{k = 1}^\infty
   {\int_{\tau-t}^\tau{\left( {a_\varepsilon
  \left(
   {\varpi_{2,k} {v^\varepsilon_2\left( s,\tau-t,0 \right)}
   ,\varpi _{2,k} {v^\varepsilon_2\left( s,\tau-t,0 \right)}} \right)
   +  a_\varepsilon  \left( {\sigma _{2,k,\varepsilon }
    ,\sigma _{2,k,\varepsilon } } \right)}
    \right )
    ds} }  \\
  &\le 2 \sum\limits_{k = 1}^\infty
   {  \varpi_{2,k} ^2 } \int_{\tau-t}^\tau {a_\varepsilon
    \left( {v^\varepsilon_2\left( s,\tau-t,0 \right),
    v^\varepsilon_2\left( s,\tau-t,0 \right)} \right)ds}  +
    \overline g |\calo|
    \sum\limits_{k = 1}^\infty
     {\left\| {\sigma_{2,k}^{*}} \right\|_{L^\infty  \left( {\widetilde \calo} \right)}^2 }t.
 \end{split}
\end{align}

By (\ref{tu1})-(\ref{tu4})   we get
\begin{align}\label{tu5}
\begin{split}
 &\E\left(a_\varepsilon  \left( {v^\varepsilon_2\left( \tau,\tau-t,0 \right),v^\varepsilon_2\left( \tau,\tau-t,0 \right)} \right)\right)\\
  &\quad  + (\gamma-2 \sum\limits_{k = 1}^\infty
   {  \varpi_{2,k} ^2 })\E\left( \int_{\tau-t}^\tau {a_\varepsilon
    \left( {v^\varepsilon_2\left( s,\tau-t,0 \right),v^\varepsilon_2\left( s,\tau-t,0 \right)}
     \right)  ds}\right)  \\
 & \leq \frac{2}{\gamma }\beta ^2 \int_{\tau-t}^\tau\E\left( {a_\varepsilon
  \left( {u^\varepsilon  \left( s,\tau-t,\xi^\varepsilon \right),u^\varepsilon  \left( s,\tau-t,\xi^\varepsilon \right)} \right)}\right) ds
  +\frac{2}{\gamma } \int_{\tau-t}^\tau {a_\varepsilon
    \left( {G_{2,\varepsilon }(s,\cdot) ,G_{2,\varepsilon }(s,\cdot) } \right)} ds\\
    &\quad+    \overline g |\calo|
    \sum\limits_{k = 1}^\infty
     {\left\| {\sigma_{2,k}^{*}} \right\|_{L^\infty  \left( {\widetilde \calo} \right)}^2t }.
 \end{split}
\end{align}
By Gronwall inequality, we have
\begin{align}
\begin{split}
 &\E\left(a_\varepsilon  \left( {v^\varepsilon_2
 \left( \tau,\tau-t,0 \right),v^\varepsilon_2\left( \tau,\tau-t,0 \right)} \right)\right)\\
 & \leq \frac{2}{\gamma }\beta ^2 \int_{\tau-t}^\tau {e^{-(\gamma-2 \sum\limits_{k = 1}^\infty
   {  \varpi_{2,k} ^2 })(\tau-s)}\E\left(a_\varepsilon
  \left( {u^\varepsilon  \left( s,\tau-t,\xi^\varepsilon \right),
  u^\varepsilon  \left( s,\tau-t,\xi^\varepsilon \right)} \right)\right)} ds\\
 &\quad +c_1e^{-(\gamma-2 \sum\limits_{k = 1}^\infty
   {  \varpi_{2,k} ^2 })\tau} \int_{-\infty}^\tau {e^{(\gamma-2 \sum\limits_{k = 1}^\infty
   {  \varpi_{2,k} ^2 })s}\|\nabla G(s,\cdot)\|^2_{L^\infty(\calo)}} ds\\
    &\quad+  \overline g |\calo|
    \sum\limits_{k = 1}^\infty
     {\left\| {\sigma_{2,k}^{*}} \right\|_{L^\infty  \left( {\widetilde \calo} \right)}^2 }
     e^{-(\gamma-2 \sum\limits_{k = 1}^\infty
   {  \varpi_{2,k} ^2 })\tau}
      \int_{-\infty}^\tau e^{(\gamma-2 \sum\limits_{k = 1}^\infty
   {  \varpi_{2,k} ^2 })s} ds,
 \end{split}
\end{align}
which together with \eqref{ee} in Lemma \ref{Tsb2} completes the proof.
\end{proof}

\begin{lem}\label{eiv1}
  Suppose \eqref{cs1}-\eqref{La} and \eqref{acg}  hold.
  Then for every $\tau\in \R$,   $T>0$, $\xi^\varepsilon  \in
L^2_{\mathcal{F}_\tau}(\Omega,   \mathbb L^2(\mathcal {O}))$,  $0<\varepsilon < \varepsilon_0$
 and $t\in [\tau,\tau+T]$,
\begin{align*}
\begin{split}
&\int_{\tau}^t {\E\left( {\left\| {u^\varepsilon  \left(s, \tau,\xi^\varepsilon \right)}
 \right\|_{H_\varepsilon ^1 \left( \calo \right)}^2+\|v^\varepsilon(s, \tau,\xi^\varepsilon)\|^2_{L^2(\calo)} } \right)} ds \le
  M_5 \E(\left\| {\xi^\varepsilon} \right\|_{\mathbb H_g(\mathcal O)}^2)\\
 &\quad+M_5 \int_{ \tau }^{\tau+T}
 \left( {1 + \left\| {G _1 \left( { s,\cdot } \right)} \right\|_{L^\infty( { {\widetilde {\mathcal O}} } )}^2 }
 +\left\| {G_2\left( { s,\cdot} \right)} \right\|_{L^\infty( { {\widetilde {\mathcal O}} } )}^2
 \right)ds,
 \end{split}
\end{align*}
where $M_5$ is a positive constant
independent of
 $\tau$, $T$ and
    $\varepsilon$.
\end{lem}

\begin{proof}
By integrating \eqref{sb5} on $(\tau,t)$,  the desired inequality
follows.
\end{proof}

\begin{lem}\label{eivv1}
  Suppose \eqref{cs1}-\eqref{La} and \eqref{acg}  hold.
  Then for every $\tau\in \R$,   $T>0$, $\xi^\varepsilon  \in
L^2_{\mathcal{F}_\tau}(\Omega,   \mathbb H^1(\mathcal {O}))$,  $0<\varepsilon < \varepsilon_0$
 and $t\in [\tau,\tau+T]$,
\begin{align*}
\begin{split}
&\int_{\tau}^t {\E\left( {\left\| {v^\varepsilon  \left(s, \tau,\xi^\varepsilon \right)}
 \right\|_{H_\varepsilon ^1 \left( \calo \right)}^2 } \right)} ds \le
 M_6 \E(\left\| {\xi^\varepsilon} \right\|_{\mathbb H^1_\varepsilon(\mathcal O)}^2)\\
 &\quad+M_6 \int_{ \tau }^{\tau+T}
 \left( {1 + \left\| {G _1 \left( { s,\cdot } \right)} \right\|_{L^\infty( { {\widetilde {\mathcal O}} } )}^2 }
 +\left\| {G_2\left( { s,\cdot} \right)} \right\|_{H^{1,\infty}( { {\widetilde {\mathcal O}} } )}^2
 \right)ds,
 \end{split}
\end{align*}
where $M_6$ is a positive constant
independent of
 $\tau$, $T$ and
    $\varepsilon$.
\end{lem}

\begin{proof}
By Ito's formula and \eqref{aeu4}, we get for $t\geq \tau$,
\begin{align}\label{tuv1}
\begin{split}
 &a_\varepsilon  \left( {v^\varepsilon\left( t \right),v^\varepsilon\left( t \right)} \right)
    + 2\gamma \int_\varsigma^\tau {a_\varepsilon
     \left( {v^\varepsilon\left( s \right),v^\varepsilon\left( s \right)}
     \right)  ds}  \\
 & =a_\varepsilon  \left( {v^\varepsilon\left( \tau \right),v^\varepsilon\left( \tau \right)} \right)
 +2\beta\int_\varsigma^\tau {\left( {u^\varepsilon(s) ,A_\varepsilon
    v^\varepsilon\left( s \right)} \right)_{H_g \left( \calo \right)} ds}
     + 2\int_\varsigma^\tau {\left( {G_{2,\varepsilon}(s,\cdot) ,A_\varepsilon
    v^\varepsilon\left( s\right)} \right)_{H_g \left( \calo \right)} ds}   \\
 &\quad + \sum\limits_{k = 1}^\infty  {\int_\varsigma^\tau {a_\varepsilon
   \left( {\sigma _{2,k,\varepsilon }
  +\varpi_{2,k} {v^\varepsilon\left( s \right)},\sigma _{2,k,\varepsilon }
   + \varpi _{2,k} {v^\varepsilon\left( s \right)}} \right)} } ds \\
  &\quad+ 2\sum\limits_{k = 1}^\infty
   {\int_\varsigma^\tau {\left( {\sigma _{2,k,\varepsilon }
    + \varpi _{2,k} {v^\varepsilon\left( s \right)},A_\varepsilon
   v^\varepsilon\left( s \right)} \right)_{H_g \left( \calo \right)}} } dW_k \left( s \right).
 \end{split}
\end{align}
Based on estimates similar to those in Lemma \ref{aTsb4}, we can infer that
\begin{align}\label{aatu5}
\begin{split}
 &\E\left(a_\varepsilon  \left( {v^\varepsilon\left( t \right),v^\varepsilon\left( t \right)} \right)\right)
   + (\gamma-2 \sum\limits_{k = 1}^\infty
   {  \varpi_{2,k} ^2 })\E\left( \int_{\tau}^t {a_\varepsilon
   \left( {v^\varepsilon\left( s\right),v^\varepsilon\left( s \right)}
     \right)  ds}\right)  \\
 & \leq \E\left(a_\varepsilon  \left( {v^\varepsilon\left( \tau \right),v^\varepsilon\left( \tau \right)} \right)\right)+
  \frac{2}{\gamma }\beta ^2 \int_{\tau}^t {a_\varepsilon
  \left( {u^\varepsilon  \left( s \right),u^\varepsilon  \left( s \right)} \right)} ds
  +\frac{2}{\gamma } \int_{\tau}^t {a_\varepsilon
    \left( {G_{2,\varepsilon }(s,\cdot) ,G_{2,\varepsilon }(s,\cdot) } \right)} ds\\
    &\quad+    \overline g |\calo|
    \sum\limits_{k = 1}^\infty
     {\left\| {\sigma_{2,k}^{*}} \right\|_{L^\infty  \left( {\widetilde \calo} \right)}^2 }(t-\tau),
 \end{split}
\end{align}
which together with  Lemma \ref{eiv1} completes the proof.

\end{proof}

The following lemmas
is concerned with the
uniform estimates of
solutions of \eqref{aeu4}
which is
similar to Lemma \ref{eiv1} and \ref{eivv1}.

\begin{lem}\label{eiv2}
  Suppose \eqref{cs1}-\eqref{La} and \eqref{acg} hold.
    Then for every $\tau\in \R$,   $T>0$ and $\xi^0  \in
L^2_{\mathcal{F}_\tau}(\Omega,   \mathbb L^2(\mathcal {Q}))$
and  $t\in [\tau,\tau+T]$,
\begin{align*}
\begin{split}
&\int_{\tau}^t {\E\left( {\left\| {u^0  \left(s, \tau,\xi^0 \right)}
 \right\|_{H ^1 \left( \mathcal Q \right)}^2 } \right)} ds \le M_7 \E(\left\| {\xi^0} \right\|_{\mathbb H_g(\mathcal Q)}^2)\\
 &\quad+M_7 \int_{ \tau }^{\tau+T}
 \left( {1 + \left\| {G _1 \left( { s,\cdot } \right)} \right\|_{L^\infty( { {\widetilde {\mathcal O}} } )}^2 }
 +\left\| {G_2\left( { s,\cdot} \right)} \right\|_{L^\infty( { {\widetilde {\mathcal O}} } )}^2
 \right)ds ,
 \end{split}
\end{align*}
where $M_7$ is a positive constant
independent of
   $\tau$ and     $T$.
\end{lem}

\begin{lem}\label{eivv10}
  Suppose \eqref{cs1}-\eqref{La} and \eqref{acg}  hold.
  Then for every $\tau\in \R$,   $T>0$, $\xi^0  \in
L^2_{\mathcal{F}_\tau}(\Omega,   \mathbb H^1(\mathcal {Q}))$
 and $t\in [\tau,\tau+T]$,
\begin{align*}
\begin{split}
&\int_{\tau}^t {\E\left( {\left\| {v^0  \left(s, \tau,\xi^0 \right)}
 \right\|_{H ^1 \left( \mathcal Q \right)}^2 } \right)} ds \le
 M_8 \E(\left\| {\xi^0} \right\|_{\mathbb H^1(\mathcal Q)}^2)\\
 &\quad+M_8 \int_{ \tau }^{\tau+T}
 \left( {1 + \left\| {G _1 \left( { s,\cdot } \right)} \right\|_{L^\infty( { {\widetilde {\mathcal O}} } )}^2 }
 +\left\| {G_2\left( { s,\cdot} \right)} \right\|_{H^{1,\infty}( { {\widetilde {\mathcal O}} } )}^2
 \right)ds,
 \end{split}
\end{align*}
where $M_8$ is a positive constant
independent of $\tau$ and $T$.
\end{lem}

\section{The Proof of Existence of Pullback Measure  Attractors}
\setcounter{equation}{0}

In this section, we prove the
  existence and uniqueness  of $\mathcal D_1$-pullback measure attractors
of \eqref{aeu3} in $\mathcal P_2(\mathbb L^2(\calo))$
and $\mathcal D_0$-pullback measure attractors of \eqref{aeu4} in
 $\mathcal P_2(\mathbb L^2(\mathcal Q))$.

By Lemma \ref{Tsb2}, we
obtain a $\mathcal D_1$-pullback absorbing set
 for $S^\varepsilon$
 as stated below.

\begin{lem}\label{Lms}
Suppose \eqref{cs1}-\eqref{La} and \eqref{acg} hold.
Given $\tau\in \R$, denote by
\be\label{Lms a}
K(\tau)
=   B_{
	{\mathcal{P}}_2
	(\mathbb L^2(\mathcal O))} \left( {L^{\frac 12} _1(\tau) } \right),
\ee
 where
\begin{align}\label{Lms b}
	\begin{split}
		L_1(\tau)=
		{M_1}
		+  {M_1}
		\int_{-\infty}^\tau e^{ - \zeta \left( {\tau - s} \right)}
		\left (
		\left\|  {G_1(s,\cdot)} \right\|_{L^\infty( { {\widetilde {\mathcal O}} } )}^2
		+
		\left\|  {G_2(s,\cdot)} \right\|^2_{L^\infty( { {\widetilde {\mathcal O}} } )}
		\right ) ds,
	\end{split}
\end{align}
and $M_1>0$ is the same constant
as in Lemma \ref{Tsb2},
  independent of $\tau$
and $\varepsilon$.
Then
$K=\{K(\tau):
\ \tau\in \R \}
\in
 \mathcal D_1$
 is a
closed  $\mathcal D_1$-pullback absorbing set
of  $S^\varepsilon$
for  $0<\varepsilon<\varepsilon_0$; that is, 
for every $\tau\in \R$ and $D_1=\{D_1(t):t\in \R\}\in \mathcal D_1$, there exists   $T=T(\tau,D_1)>0$,
		independent of $\varepsilon$, such that for  all
		$t\geq T$ and $0<\varepsilon<\varepsilon_0$,     
	\be
		S^\varepsilon\left( t,\tau-t \right)D_1(\tau-t)
		\subseteq K(\tau).
	\ee
\end{lem}

{\begin{proof}
		By \eqref{Lms a}  and
		 Lemma
		\ref{Tsb2}, we see that
		for every $\tau\in \R$ and $D_1=\{D_1(t):t\in \R\}\in \mathcal D_1$, there exists   $T=T(\tau,D_1)>0$,
		independent of $\varepsilon$, such that for  all
		$t\geq T$ and $0<\varepsilon<\varepsilon_0$,     $S^\varepsilon$    satisfies
	\be\label{Lms p1}
		S^\varepsilon\left( t,\tau-t \right)D_1(\tau-t)
		\subseteq K(\tau).
	\ee
	We now prove
	$K=\{K(\tau):
	\ \tau\in \R \}
	\in
	\mathcal D_1$.
By \eqref{Lms a}, \eqref{Lms b}
and \eqref{acg}   we have
	$$
	e^{\zeta \tau}\| K(\tau) \|^2
	_{{\mathcal{P}}_2
		(L^2(\mathcal O))}=
		e^{\zeta \tau}
		L_1(\tau)
	$$
	$$
	=
	{ {e^{\zeta \tau} M_1}}
	+  {  { M_1}}
	\int_{-\infty}^\tau e^{   \zeta  s }
	\left (
	\left\|  {G_1(s,\cdot)} \right\|_{L^\infty( { {\widetilde {\mathcal O}} } )}^2
	+
	\left\|  {G_2(s,\cdot)} \right\|^2_{L^\infty( { {\widetilde {\mathcal O}} } )}
	\right ) ds
	\to 0, \ \ \text{as } \tau\to -\infty,
	$$
	and hence
		$K=\{K(\tau):
	\ \tau\in \R \}
	\in
	\mathcal D_1$,
	which along with
	\eqref{Lms p1} concludes
	the proof.
	\end{proof}

Let $\chi^n=\chi_1^n+\chi_2^n$, for $n\in \N$, where
$\{\chi^n\}_{n=1}^\infty$, $\{\chi^n_1\}_{n=1}^\infty$
and $\{\chi^n_2\}_{n=1}^\infty$ are $X$-valued random sequence.

\begin{lem}\label{Lch}
Assume  $\{\mathcal L(\chi_1^n)\}_{n=1}^\infty$ are tight and $\E(\|\chi_2^n\|_X^2)$
convergence to zero. Then $\{\mathcal L(\chi^n)\}_{n=1}^\infty$ are also tight.
\end{lem}

\begin{proof}
Since $\{\mathcal L(\chi_1^n)\}_{n=1}^\infty$ are tight, there exists
a subsequence of $\{\mathcal L(\chi_1^n)\}_{n=1}^\infty$ (not relabelled)
and a $\mu\in \mathcal P(X)$
such that $\{\mathcal L(\chi_1^n)\}_{n=1}^\infty$ convergence weakly  to $\mu$, i.e.,
for any $\varphi\in C_b(X)$,
\be\label{chi1}
(\mathcal L(\chi_1^n), \varphi)\rightarrow (\mu,\varphi), \quad \text{as}\quad n\rightarrow\infty.
\ee
Meanwhile,
\begin{align}\label{chi2}
\begin{split}
 &\mathop {\lim }\limits_{\scriptstyle \varphi  \in L_b (X) \hfill \atop
  \scriptstyle \left\| \varphi  \right\|_L  \le 1 \hfill} \left| {\E\left( {\varphi \left( {\chi ^n }
   \right)} \right) - \E\left( {\varphi \left( {\chi _1^n } \right)} \right)} \right| \\
 & \le \mathop {\lim }\limits_{\scriptstyle \varphi  \in L_b (X) \hfill \atop
  \scriptstyle \left\| \varphi  \right\|_L  \le 1 \hfill} \E\left( {\left|
  {\left( {\varphi \left( {\chi ^n } \right)} \right) -
  \left( {\varphi \left( {\chi _1^n } \right)} \right)} \right|} \right) \\
  &\le \E\left( {\left\| {\chi ^n  - \chi _1^n } \right\|_X } \right) =
  \E\left( {\left\| {\chi _2^n } \right\|_X } \right) \le
   \E\left( {\left\| {\chi _2^n } \right\|_X^2 } \right) \to 0.
\end{split}
 \end{align}
By \eqref{chi1} and \eqref{chi2} we have for any $\varphi\in C_b(X)$,
\[
(\mathcal L(\chi^n), \varphi)\rightarrow (\mu,\varphi), \quad \text{as}\quad n\rightarrow\infty.
\]
This completes the proof.

\end{proof}

We now present  the $\mathcal D_1$-pullback asymptotical
 compactness  of $S^\varepsilon$ associated with   \eqref{aeu3}.

\begin{lem}\label{Ltc}
If  \eqref{cs1}-\eqref{La} and $\eqref{acg}$ hold, then
$S^\varepsilon$ with $0<\varepsilon<\varepsilon_0$  is $\mathcal D_1$-pullback asymptotically compact
 in $\mathcal P_2 \left(\mathbb L^2(\mathcal O) \right)$;
  that is,
for every $\tau\in \R$, $\left\{ {S^\varepsilon  \left( {t_n,\tau-t_n} \right)\mu _n } \right\}_{n = 1}^\infty
$ has a convergent subsequence in $\mathcal P_2 \left(\mathbb L^2(\mathcal O) \right)$ whenever
$t_n\rightarrow +\infty$ and $\mu _n\in D_1(\tau-t_n)$ with $D_1\in \mathcal D_1$.
\end{lem}
\begin{proof}
Given $s\in \R$, $t\geq s$ and $\xi^\varepsilon=(\xi^\varepsilon_1,\xi^\varepsilon_2)\in
 L^2_{\mathcal{F}_\tau}( \mathbb L^2(\mathcal {O}) ))$,
define
$$w^\varepsilon_1(t,s,\xi^\varepsilon)=(0,v^\varepsilon_1(t,s,\xi^\varepsilon_2))\quad
\text{and}\quad w^\varepsilon_2(t,s,\xi^\varepsilon)
=(u^\varepsilon(t,s,\xi^\varepsilon),v^\varepsilon_2(t,s,\xi^\varepsilon)).$$
It is clear that
\be \label{w12}
w\left( {t,s,\xi ^\varepsilon  } \right) = w_1 \left( {t,s,\xi ^\varepsilon  } \right)
+ w_2 \left( {t,s,\xi ^\varepsilon  } \right).
\ee

Let $p^{1,\varepsilon}(t,s)$ and $p^{2,\varepsilon}(t,s)$ be the transition process associated with
$w^\varepsilon_1$  and $w^\varepsilon_2$, respectively,
which are similar to the definition of $p^{\varepsilon}(t,s)$.
Define $p_*^{1,\varepsilon}(t,s)$ and $p_*^{2,\varepsilon}(t,s)$
are the dual operators of $p^{1,\varepsilon}(t,s)$ and $p^{2,\varepsilon}(t,s)$,
respectively.

Given
  $t\in \R^+$, $\tau\in \R$, $i=1,2$ and
  $0<\varepsilon<\varepsilon_0$,
  let $S^{i,\varepsilon}(t,\tau):\mathcal P_2(\mathbb L^2(\calo))\rightarrow\mathcal P_2(\mathbb L^2(\calo))$
  be the map given by
$$
S^{i,\varepsilon}(t,\tau)\mu
=p_{\ast}^{i,\varepsilon}(\tau+t,\tau)
\mu ,
\quad \forall  \
 \mu \in \mathcal P_2(\mathbb L^2(\calo)).
$$

By \eqref{v1e} we get
\be \label{s1to}
\E(\|w_1 \left( {\tau-t,\tau,\xi ^\varepsilon  } \right)\|^2)\rightarrow 0, \quad \text{as}\quad t\rightarrow\infty.
\ee
Form \eqref{w12}, \eqref{s1to} and Lemma \ref{Lch} it follows that the sequence
 $\left\{ {S^\varepsilon  \left( {t_n,\tau-t_n} \right)\mu _n } \right\}_{n = 1}^\infty
$ has a convergent subsequence in $\mathcal P_2 \left(\mathbb  L^2(\mathcal O) \right)$  as long as
 $\left\{ {S^{2,\varepsilon}  \left( {t_n,\tau-t_n} \right)\mu _n } \right\}_{n = 1}^\infty
$ is tight.

Given $R>0$ and $0<\varepsilon<\varepsilon_0$, let
$  B_\varepsilon(R)=\{ (u,v)\in \mathbb H^1(\calo): \| (u,v) \|_{\mathbb H^1_\varepsilon(\calo) } \le R\}$.
By Lemma
\ref{Tsb4} and \ref{aTsb4} we see that  for every $\tau\in \R$
 and $D_1=\{D_1(t):t\in \R\}\in \mathcal D_1$, there exists   $T=T(\tau,D_1)>0$,
independent of $\varepsilon$, such that for  all
$t\geq T$ and $0<\varepsilon<\varepsilon_0$, the solution $u^\varepsilon$ of \eqref{aeu3}
and $v^\varepsilon$ of \eqref{v21}
 satisfies
$$
\E \left (
\|u^\varepsilon \left( {\tau,\tau-t,\xi^\varepsilon} \right)\|^2_{H^1_\varepsilon(\calo)}
+\|v_2^\varepsilon \left( {\tau,\tau-t,\xi^\varepsilon} \right)\|^2_{H^1_\varepsilon(\calo)}
\right ) \le L_2(\tau),
$$
where
 $\xi^\varepsilon  \in
L^2_{\mathcal{F}_{\tau-t}}(\Omega,   L^2(\mathcal {O}))$
with
$\mathcal L(\xi^\varepsilon)\in D_1(\tau-t)$ and
 $L_2(\tau)>0$ is a   constant depending    on $\tau$, but independent
of   $D_1$ and $\varepsilon$.
Then by  Chebyshev's inequality,
we see that
for all  $R >0$,
  $t\geq T $ and $0<\varepsilon<\varepsilon_0$,
$$
P \left ( \left \{
(u^\varepsilon \left( {\tau,\tau-t,\xi^\varepsilon} \right),v^\varepsilon_2 \left( {\tau,\tau-t,\xi^\varepsilon} \right))
\in   B_\varepsilon(R ) \right \}
\right )
\ge 1-{\frac {L_2(\tau)}{R^2 }},
$$
where  $\xi^\varepsilon  \in
L^2_{\mathcal{F}_{\tau-t}}(\Omega,  \mathbb L^2(\mathcal {O}))$
with
$\mathcal L(\xi^\varepsilon)\in D_1(\tau-t)$.
This along with
the
compactness of
the   embedding
$\mathbb H^1(\calo)\hookrightarrow \mathbb L^2(\calo)$
implies
 the  tightness of  probability  distributions of the family \\
$\left\{ {(u^\varepsilon \left( {\tau,\tau-t,\xi^\varepsilon} \right),v^\varepsilon_2 \left( {\tau,\tau-t,\xi^\varepsilon} \right))}:
   \mathcal L(\xi^\varepsilon)\in D_1(\tau-t),
   \ t\geq T  \right\}$
in  $\mathbb L^2(\calo)$. Thus,
$\left\{ {S^{2,\varepsilon}  \left( {t_n,\tau-t_n} \right)\mu _n } \right\}_{n = 1}^\infty
$ is tight. This completes the proof.
  \end{proof}

Now, we are in a position to proof the  existence, uniqueness and periodicity of pullback
measure attractor.\\
 {\bf{Proof of Theorem \ref{Tie1}}.}
It follows from Lemma \ref{dy} that $S^\varepsilon$ is
a
continuous  non-autonomous dynamical system on $\mathcal P_2(\mathbb L^2(\mathcal O))$.
Notice that $S^\varepsilon$ has a $\mathcal D_1$-pullback  absorbing
set $K$ in $\mathcal P_2(\mathbb L^2(\calo))$ by Lemma \ref{Lms}, and is $\mathcal D_1$-pullback
 asymptotically compact in $P_2(\mathbb L^2(\calo))$ by Lemma \ref{Ltc}.
 Hence the existence and uniqueness  of  the  $\mathcal D_1$-pullback
measure attractor for $S^\varepsilon$ follows from Proposition  \ref{Tesa} immediately.
We now consider the periodicity of the measure attractor $\mathcal A_\varepsilon$.
By \eqref{Lms a} and \eqref{Lms b} we find that
$K$ is $\rho$-periodic. In addition, it follows from  Lemma \ref{markovp2} and \eqref{sp}
the non-autonomous dynamical system $S^\varepsilon$
 associated with system
\eqref{aeu3} is also $\varpi$-periodic. Thus, from Proposition  \ref{Tesa}, the periodicity
of the measure attractor $\mathcal A_\varepsilon$ follows.

\section{Upper Semicontinuity   of   Pullback Measure Attractors  }
\setcounter{equation}{0}

In this section, we prove the upper semicontinuity of pullback measure attractors for the
non-autonomous stochastic FitzHugh-Nagumo systems    when the $(n + 1)$-dimensional thin
domains collapse to an $n$-dimensional domain. To that end, we need the average
operator
${\mathcal M}: \mathbb L^2(\mathcal {O})
\to \mathbb L^2
(\mathcal {Q})$
as given by:
for every
$\varphi=(\varphi_1,\varphi_2)
\in \mathbb L^2(\mathcal {O})$,
$$
{\mathcal M}
\varphi (y^*)
=(\int_0^1
\varphi_1 (y^*, y_{n+1})
d y_{n+1},\int_0^1
\varphi_2 (y^*, y_{n+1})
d y_{n+1}),
\quad y^*\in {
\mathcal {Q}}.
$$

The following result on   average of  functions
 can be proved similarly as the lemma 3.1     in \cite{hale1992reaction}.

\begin{lem} \label{aaL5.3}
If $w\in \mathbb H^1(\mathcal O)$, then $\mathcal Mw\in \mathbb H^1(\mathcal Q)$ and
\[
\left\| {w - \mathcal Mw} \right\|_{\mathbb H_g \left(\mathcal O \right)}
 \le c\varepsilon \left\| w \right\|_{\mathbb H_\varepsilon ^1 \left(\mathcal O \right)},
\]
where $c$ is a constant   independent of $\varepsilon$.
\end{lem}

In the sequel, we further assume the functions $G_i$,  $F$ and $\sigma_{i,k}$, for $i=1,2$ and $k\in \N$, satisfy
 \begin{equation}\label{cg}
\left\| {G_{i,\varepsilon}(t,\cdot)
  - G_{i,0}(t,\cdot) } \right\|_{L^2 \left( \mathcal O  \right)}
  \le \kappa_1(t)\varepsilon,\quad \text{for all}\,\, t\in {\mathbb{R}},
\end{equation}
\begin{equation}\label{cF}
 \left\| {F_\varepsilon(\cdot,s)   - F_0(\cdot,s) } \right\|_{L^2 \left( \mathcal O  \right)}
  \le \kappa_2\varepsilon ,\quad \text{for all}\,\, s\in {\mathbb{R}},
\end{equation}
and
 \begin{equation}\label{cs}
\left\| {\sigma_{i,k,\varepsilon}
  - \sigma_{i,k,0} } \right\|_{L^2 \left( \mathcal O  \right)}
  \le \varsigma_k\varepsilon,
\end{equation}
where $\kappa_1\in L_{loc}^2(\R)$, $\kappa_2$ and $\varsigma_k$, $k\in \N$, are positive constants and $
\sum\limits_{k = 1}^\infty  {\varsigma _k^2 }  < \infty.
$
 By \eqref{Ftof} and \eqref{cF} we have
\begin{equation}\label{cf}
\left\| {f_\varepsilon(\cdot,s)   - f_0(\cdot,s) } \right\|_{L^2 \left( \mathcal O  \right)}
  \le \kappa_2\varepsilon ,\quad \text{for all}\,\, s\in {\mathbb{R}}.
\end{equation}

The next lemma is concerned with the convergence of
the solutions of \eqref{aeu3} as $\varepsilon \to 0$.

\begin{thm}\label{Tee}
Suppose \eqref{cs1}-\eqref{La}, \eqref{acg} and \eqref{cg}-\eqref{cs} hold.
Given $\tau\in \R$ and     $R>0$, we have  for all  $t \ge 0$,
\[\mathop {\lim }\limits_{\varepsilon  \to 0}\mathop {\sup }\limits_{\mu^\varepsilon\in B_\varepsilon ^1 \left( R \right)}
d_{ {\mathcal P_2(\mathbb L^2 \left( {\mathcal O}  \right))}}
\left (
S^\varepsilon(t,\tau)
	\mu^\varepsilon , \ \
\left (
S^0(t,\tau) (\mu^\varepsilon
\circ \mathcal   M ^{-1}
) \right) \circ
{\mathcal{I}}^{-1}
\right ) =0,\]
where $ B_\varepsilon ^1 \left( R \right) = \left\{ {\mu  \in {\mathcal P }
_2(\mathbb L^2(\mathcal{O}))} :\   \int_{\mathbb H_\varepsilon ^1 \left( \mathcal O \right)} {\left\|
\xi   \right\|_{\mathbb H_\varepsilon ^1
\left( \mathcal O \right)}^2 \mu \left( {d\xi  } \right)\leq R}  \right\}.
$
\end{thm}

\begin{proof}
Let $(u^\varepsilon(t),v^\varepsilon(t))=(u^\varepsilon(t,\tau,\xi^\varepsilon),v^\varepsilon(t,\tau,\xi^\varepsilon))$,
where $\E(\left\| \xi  \right\|_{\mathbb H_\varepsilon ^1
\left( \mathcal O \right)}^2)\leq
R$, and $(u^0(t),v^0(t))=(u^0(t,\tau,\mathcal M \xi^\varepsilon),v^0(t,\tau,\mathcal M \xi^\varepsilon))$.
By \eqref{aeu3}-\eqref{aeu4} and  Ito's formula,
we get  for $t\geq \tau$,
\begin{align}\label{uc2}
\begin{split}
&\beta\left\| {{u^\varepsilon }\left( t \right) - {u^0}\left( t \right)} \right\|_{{H_g}\left( \calo \right)}^2
+\alpha\left\| {{v^\varepsilon }\left( t \right) - {v^0}\left( t \right)} \right\|_{{H_g}\left( \calo \right)}^2\\
 &= \beta\left\| {{\xi ^\varepsilon_1 } - {\xi ^0_1}} \right\|_{{H_g}\left( \calo \right)}^2
 +\alpha\left\| {{\xi ^\varepsilon_2 } - {\xi ^0_2}} \right\|_{{H_g}\left( \calo \right)}^2\\
  &\quad- 2\beta\int_\tau^t {{{\left( {{A_\varepsilon }{u^\varepsilon }\left( s \right) -
  {A_0}{u^0}\left( s \right),{u^\varepsilon }\left( s \right) - {u^0}\left( s \right)} \right)}_{{H_g}\left( \calo \right)}}ds} \\
 &\quad- 2\beta\lambda \int_\tau^t {{{\| {{u^\varepsilon }\left( s \right) -
 {u^0}\left( s \right)} \|}^2_{{H_g}\left( \calo\right)}}ds}- 2\alpha\gamma \int_\tau^t {{{\| {{v^\varepsilon }\left( s \right) -
 {v^0}\left( s \right)} \|}^2_{{H_g}\left( \calo\right)}}ds}\\
  &\quad + 2\beta\int_\tau^t {{{\left( {{G_{1,\varepsilon}(s,\cdot) }
 - {G_{1,0}(s,\cdot)},{u^\varepsilon }\left( s \right)
 - {u^0}\left( s \right)} \right)}_{{H_g}\left( \calo \right)}}} ds\\
 &\quad+ 2\alpha\int_\tau^t {{{\left( {{G_{2,\varepsilon}(s,\cdot) }
 - {G_{2,0}(s,\cdot)},{v^\varepsilon }\left( s \right)
 - {v^0}\left( s \right)} \right)}_{{H_g}\left( \calo \right)}}} ds \\
 &\quad+ 2\beta\int_\tau^t {{{\left( {{f_\varepsilon }\left( {\cdot,{u^\varepsilon }\left( s \right)} \right)
 - {f_0}\left( {\cdot,{u^\varepsilon }\left( s \right)} \right),{u^\varepsilon }\left( s \right)
 - {u^0}\left( s \right)} \right)}_{{H_g}\left( \calo \right)}}} ds\\
&\quad +\beta \sum\limits_{k = 1}^\infty  {\int_\tau^t {\left\|
 {{\sigma _{1,k,\varepsilon }}
- {\sigma _{1,k,0}} + {
 {\varpi _{1,k}}\left( {{u^\varepsilon }\left( s \right)} \right) -
{\varpi _{1,k}}\left( {{u^0}\left( s \right)} \right)} } \right\|_{{H_g}\left( \calo \right)}^2ds} } \\
&\quad +\alpha \sum\limits_{k = 1}^\infty  {\int_\tau^t {\left\|
 {{\sigma _{2,k,\varepsilon }}
- {\sigma _{2,k,0}} + {
 {\varpi _{2,k}}{{v^\varepsilon }\left( s \right)} -
{\varpi _{2,k}} {{v^0}\left( s \right)}} } \right\|_{{H_g}\left( \calo \right)}^2ds} }\\
 &\quad+ 2\beta\sum\limits_{k = 1}^\infty
 {\int_\tau^t {{{\left( {{\sigma _{1,k,\varepsilon }}
  - {\sigma _{1,k,0}} ,{u^\varepsilon }\left( s \right) -
  {u^0}\left( s \right)} \right)}_{{H_g}\left(\calo \right)}}d{W_k}\left( s \right)} }\\
  &\quad+ 2\alpha\sum\limits_{k = 1}^\infty
 {\int_\tau^t {{{\left( {{\sigma _{2,k,\varepsilon }}
  - {\sigma _{2,k,0}} ,{v^\varepsilon }\left( s \right) -
  {v^0}\left( s \right)} \right)}_{{H_g}\left(\calo \right)}}d{W_k}\left( s \right)} }\\
&\quad+ 2\beta\sum\limits_{k = 1}^\infty  {\int_\tau^t {{{\left( {
{\varpi _{1,k}}\left( {{u^\varepsilon }\left( s \right)} \right) -
{\varpi _{1,k}}\left( {{u^0}\left( s \right)} \right),{u^\varepsilon }\left( s \right)
 - {u^0}\left( s \right)} \right)}_{{H_g}\left( \calo \right)}}d{W_k}\left( s \right)}.}
  \\
&\quad+ 2\alpha\sum\limits_{k = 1}^\infty  {\int_\tau^t {{{\left( {
{\varpi _{2,k}}{{v^\varepsilon }\left( s \right)} -
{\varpi _{2,k}} {{v^0}\left( s \right)},{v^\varepsilon }\left( s \right)
 - {v^0}\left( s \right)} \right)}_{{H_g}\left( \calo \right)}}d{W_k}\left( s \right)}.}
 \end{split}
\end{align}
For the third term on the right-hand side of \eqref{uc2}, we have
\begin{align}\label{uc3}
\begin{split}
 &- 2\beta\int_\tau^t {{{\left( {{A_\varepsilon }{u^\varepsilon }
 \left( s \right) -
  {A_0}{u^0}\left( s \right),{u^\varepsilon }\left( s \right) - {u^0}\left( s \right)} \right)}_{{H_g}\left( \calo \right)}}ds} \\
 &=  - 2\beta\int_\tau^t {{a_\varepsilon }\left( {{u^\varepsilon }\left( s \right) -
 {u^0}\left( s \right),{u^\varepsilon }\left( s \right) - {u^0}\left( s \right)} \right)ds} \\
 &\quad+ 2\beta\sum\limits_{i = 1}^n {\int_\tau^t {{{\left( {\frac{{{g_{{y_i}}}}}{g}u_{{y_i}}^0
 \left( s \right),{y_{n + 1}}\left( {u_{{y_{n + 1}}}^\varepsilon \left( s \right) -
  u_{{y_{n + 1}}}^0\left( s \right)} \right)} \right)}_{{H_g}\left( \calo \right)}}ds.} }
 \end{split}
\end{align}
By   (\ref{c3.3}) and (\ref{cf}) we obtain
\begin{align}\label{uc4}
\begin{split}
&2\beta\int_\tau^t{\left( {{f_\varepsilon }\left( {\cdot,{u^\varepsilon }\left( s \right)} \right) -
{f_0}\left( {{\cdot},{u^\varepsilon }\left( s \right)} \right),{u^\varepsilon }\left( s \right)
 - {u^0}\left( s \right)} \right)_{{H_g}\left( \calo \right)}}ds\\
 &\le 2\beta\varrho_1 \int_\tau^t\left\| {{u^\varepsilon }\left( s \right) - {u^0}\left( s \right)} \right\|_{{H_g}
 \left( \calo \right)}^2ds + c_1{\kappa _2}\varepsilon\int_\tau^t \left( {1 + 2\left\| {{u^\varepsilon }
 \left( s \right)} \right\|_{{H_g}\left( \calo \right)}^2 + 2\left\| {{u^0}\left( s \right)}
  \right\|_{{H_g}\left( \calo \right)}^2} \right)ds.
 \end{split}
\end{align}
By  (\ref{cg}), we  get
\begin{align}\label{uc5}
\begin{split}
 &2\beta\int_\tau^t {\left( {G_{1,\varepsilon}(s,\cdot)   - G_{1,0}(s,\cdot)  , u^\varepsilon
   \left( s \right) - u^0 \left( s \right)} \right)_{H_g
    \left( \calo \right)} } ds\\
  &\le c_2 \varepsilon \int_\tau^t {\left( {\kappa^2_1(s) + \left\| {u^\varepsilon  \left( s \right)}
  \right\|_{H_g \left(  \calo  \right)}^2  + \left\| {u^0 \left( s \right)} \right\|_{H_g \left( \mathcal Q   \right)}^2
  + \left\| {v^\varepsilon  \left( s \right)}
  \right\|_{H_g \left(  \mathcal O \right)}^2  +
  \left\| {v^0 \left( s \right)} \right\|_{H_g \left( \mathcal Q  \right)}^2 } \right)} ds.
 \end{split}
\end{align}
By  \eqref{cs}, we  get
\begin{align}\label{uc6}
\begin{split}
2\beta\sum\limits_{k = 1}^\infty  {\int_\tau^t {\left\|
 {{\sigma _{1,k,\varepsilon }}  -
{\sigma _{1,k,0}} } \right\|_{{H_g}\left( \calo \right)}^2dt} }
+2\alpha\sum\limits_{k = 1}^\infty  {\int_\tau^t {\left\|
 {{\sigma _{2,k,\varepsilon }}  -
{\sigma _{2,k,0}} } \right\|_{{H_g}\left( \calo \right)}^2dt} }
 \le {c _3}{\varepsilon ^2}(t-\tau).
 \end{split}
\end{align}
It follows from \eqref{cv3} we  get
\begin{align}\label{uc7}
\begin{split}
&2\beta\sum\limits_{k = 1}^\infty  {\int_\tau^t {\left\| {{\varpi _{1,k}}
\left( {{u^\varepsilon }\left( s \right)} \right) - {\varpi _{1,k}}
\left( {{u^0}\left( s \right)} \right)} \right\|_{{H_g}\left( \calo \right)}^2dt} }\\
&\quad+2\alpha\sum\limits_{k = 1}^\infty  {\int_\tau^t {\left\| {{\varpi _{2,k}}
{{v^\varepsilon }\left( s \right)} - {\varpi _{2,k}}
{{v^0}\left( s \right)}} \right\|_{{H_g}\left( \calo \right)}^2dt} }\\
& \le  \sum\limits_{k = 1}^\infty  (2\beta L_{k}^2+2\alpha \varpi_{2,k}^2)
 \int_\tau^t {\left(\left\| {{u^\varepsilon }\left( s \right) - {u^0}\left( s \right)} \right\|_{{H_g}\left( \calo \right)}^2
 +\left\| {{v^\varepsilon }\left( s \right) - {v^0}\left( s \right)} \right\|_{{H_g}\left( \calo \right)}^2\right)dt}.
 \end{split}
\end{align}
Finally, by (\ref{a4.2}), we have
\begin{align}\label{uc6}
\begin{split}
&2\beta \sum\limits_{i=1}^{n} { \int_{\tau}^t\left( {\frac{{g_{y_i} }}{g} u^0 _{y_i }(s),y_{n+1}
  ( u^\varepsilon_{y_{n+1}}(s) -u^0_{y_{n+1}}(s)  )
    } \right)_{H_g({\mathcal O}) }ds}\\
   &=2\beta\sum\limits_{i=1}^{n}{ \int_{\tau}^t\left( {g_{y_i}  u^0_{y_i }(s),
   y_{n+1}  ( u^\varepsilon_{y_{n+1}}(s) -u^0_{y_{n+1}}(s)  )
   } \right)_{L^2(\mathcal O)}ds}\\
   &\le c_4\varepsilon\int_{\tau}^t( \left\| u^0(s) \right\|_{H^1 \left({\mathcal Q} \right)}
    \left\| {u^\varepsilon(s) -u^0(s)  } \right\|_{H_\varepsilon ^1\left( {\mathcal O} \right) })ds\\
& \le  c_5\varepsilon \int_{\tau}^t\left( {\left\| {u^\varepsilon(s)  } \right\|_{H_\varepsilon ^1
   \left( {\mathcal O} \right)}^2
   + \left\| u^0(s) \right\|_{H ^1 \left( {\mathcal Q} \right)}^2 } \right)ds.
 \end{split}
\end{align}

Taking the expectation of \eqref{uc2} and using \eqref{uc3}-\eqref{uc6}, we obtain for all $t\geq \tau$,
\begin{align}\label{uc7}
\begin{split}
 &\E\left( {\left\| {u^\varepsilon  \left( t \right) - u^0 \left( t \right)} \right\|_{H_g \left( \calo \right)}^2
 +\left\| {v^\varepsilon  \left( t \right) - v^0 \left( t \right)} \right\|_{H_g \left( \calo \right)}^2 }
  \right) \\
  &\le c_6\left\| {\xi ^\varepsilon   -\calm \xi ^\varepsilon }
   \right\|_{\mathbb H_g \left( \calo \right)}^2 \\
 &\quad +c_7\int_\tau^t {\E\left( {\left\| {u^\varepsilon
  \left( s \right) - u^0 \left( s \right)} \right\|_{H_g \left( \calo \right)}^2+
  \left\| {v^\varepsilon
  \left( s \right) - v^0 \left( s \right)} \right\|_{H_g \left( \calo \right)}^2 } \right)} ds \\
  &\quad + c_8\varepsilon \int_\tau^t {\E\left( {1+\kappa^2_1(s)+\left\| {u^\varepsilon  \left( s \right)}
   \right\|_{H_\varepsilon ^1 \left( \calo \right)}^2  + \left\| {u^0 \left( s \right)}
    \right\|_{H^1 \left( \mathcal Q \right)}^2+\left\| {v^\varepsilon  \left( s \right)}
   \right\|_{H ^1 \left( \mathcal O \right)}^2  + \left\| {v^0 \left( s \right)}
    \right\|_{H^1 \left( \mathcal Q \right)}^2 } \right)} ds.
 \end{split}
\end{align}
By \eqref{uc7}, Lemma \ref{Tsb4} and Lemma \ref{eiv1}-\ref{eivv10} we find that for $t\geq \tau$,
\begin{align*}
\begin{split}
 &\E\left( {\left\| {u^\varepsilon  \left( t \right) - u^0 \left( t \right)} \right\|_{H_g \left( \calo \right)}^2
 +\left\| {v^\varepsilon  \left( t \right) - v^0 \left( t \right)} \right\|_{H_g \left( \calo \right)}^2 } \right)\\
& \le c_6\left\| {\xi ^\varepsilon   -\calm \xi ^\varepsilon } \right\|_{\mathbb H_g \left( \calo \right)}^2
+c_7\int_\tau^t {\E\left( {\left\| {u^\varepsilon  \left( s \right) - u^0 \left( s \right)}
  \right\|_{H_g \left( \calo \right)}^2+
  \left\| {v^\varepsilon
  \left( s \right) - v^0 \left( s \right)} \right\|_{H_g \left( \calo \right)}^2 } \right)} ds\\
 &\quad + c_9\varepsilon  +
  c_9\varepsilon (t-\tau).
 \end{split}
\end{align*}
Then by Gronwall's inequlity and  Lemma \ref{aaL5.3},
we infer that for every $T>0$ and    $t \in [\tau, \tau+T]$,
\begin{align}\label{uc8}
\begin{split}
& \E\left( {\left\| {u^\varepsilon  \left( t \right) - u^0
  \left( t \right)} \right\|_{H_g \left( \calo \right)}^2+\left\| {v^\varepsilon  \left( t \right) - v^0
  \left( t \right)} \right\|_{H_g \left( \calo \right)}^2 } \right)\\
& \le \left( {c_6\left\| {\xi ^\varepsilon   -\calm \xi ^\varepsilon } \right\|_{\mathbb H_g \left( \calo \right)}^2
  + c_9\varepsilon
 +c_9\varepsilon T } \right)e^{ c_7t}  \\
 & \le \left( {c_{10}\varepsilon \left\| {\xi ^\varepsilon  }
  \right\|_{\mathbb H_\varepsilon ^1 \left( \calo \right)}^2  +
  c_9\varepsilon+ c_9\varepsilon T  } \right)e^{ c_7t}.
 \end{split}
\end{align}
It follows from  \eqref{uc8} that
\begin{align}\label{uc9}
\mathop {\lim }\limits_{\varepsilon  \to 0} \mathop {\sup }\limits_{\E\left({{\left\| {{\xi ^\varepsilon }}
 \right\|}_{\mathbb H_\varepsilon ^1\left( \calo\right)}^2}\right) \le R}\E\left( \left\| {{w^\varepsilon }
 \left( {t,\tau,{\xi ^\varepsilon }} \right) - {w^0}\left( {t,\tau,\calm{\xi ^\varepsilon }}
  \right)} \right\|_{\mathbb L^2(\mathcal O)}^2\right) = 0.
\end{align}

 Note that for
 all
 $t\geq \tau$ we have
 $$
 \mathop {\sup }\limits_{\E
 ({{\left\| {{\xi ^\varepsilon }}
 \right\|}_{\mathbb H_\varepsilon ^1\left( \calo\right)}^2}) \le R}\ \
\mathop {\sup }\limits_{\scriptstyle \varphi \in L_b \left( \mathbb L^2(\calo) \right) \hfill \atop
  \scriptstyle \left\| \varphi \right\|_L  \le 1 \hfill}
\left| {\E
\left (
\varphi\left( {w^\varepsilon \left( {t,\tau,\xi^\varepsilon } \right)} \right)
\right )
- \E
\left (
 \varphi\left( {w^0 \left( {t,\tau,\mathcal M\xi^\varepsilon } \right)} \right )
 \right ) }
 \right|
$$
$$
\le
 \mathop {\sup }\limits_{\E
 ({{\left\| {{\xi ^\varepsilon }}
 \right\|}_{\mathbb H_\varepsilon ^1\left( \calo\right)}^2}) \le R}\ \
\mathop {\sup }\limits_{\scriptstyle \varphi \in L_b \left(\mathbb L^2(\calo) \right) \hfill \atop
  \scriptstyle \left\| \varphi \right\|_L  \le 1 \hfill}
 {\E
\left ( \left |
\varphi\left( {w^\varepsilon \left( {t,\tau,\xi^\varepsilon } \right)} \right )
-
 \varphi\left( {w^0 \left( {t,\tau,\mathcal M\xi^\varepsilon } \right)} \right )
 \right |
 \right ) }
$$
 $$
\le
 \mathop {\sup }\limits_{\E
 ({{\left\| {{\xi ^\varepsilon }}
 \right\|}_{\mathbb H_\varepsilon ^1\left( \calo\right)}^2}) \le R}
 {\E
\left (
\|  {w^\varepsilon \left( {t,\tau,\xi^\varepsilon } \right)}
-
 {w^0 \left( {t,\tau,\mathcal M\xi^\varepsilon } \right) }
 \|_{\mathbb L^2(\mathcal{O})}
 \right ) }
$$
$$
\le
\left (
 \mathop {\sup }\limits_{\E
 ({{\left\| {{\xi ^\varepsilon }}
 \right\|}_{\mathbb H_\varepsilon ^1\left( \calo\right)}^2}) \le R}
 {\E
\left (
\|  {w^\varepsilon \left( {t,\tau,\xi^\varepsilon } \right)}
-
 {w^0 \left( {t,\tau,\mathcal M\xi^\varepsilon } \right) }
 \|^2 _{\mathbb L^2(\mathcal{O})}
 \right ) }
 \right )^{\frac 12}
$$
which along with
\eqref{uc9} implies that  for all $t\ge \tau$,
\[
\lim_{\varepsilon
\to 0} \  \
\mathop {\sup }\limits_{\E({{\left\| {{\xi ^\varepsilon }}
 \right\|}_{\mathbb H_\varepsilon ^1\left( \calo\right)}^2}) \le R}
 \   \
\mathop {\sup }\limits_{\scriptstyle \varphi \in L_b \left( \mathbb L^2(\calo) \right) \hfill \atop
  \scriptstyle \left\| \varphi \right\|_L  \le 1 \hfill}
\left| {\E
\left (
\varphi\left( {w^\varepsilon \left( {t,\tau,\xi^\varepsilon } \right)} \right)
\right )
- \E
\left (
 \varphi\left( {w^0 \left( {t,\tau,\mathcal M\xi^\varepsilon } \right)} \right )
 \right ) }
 \right| =0.
\]
This completes the proof.
\end{proof}

To examine the limiting behavior of the family of
attractors  $\mathcal{A}_\varepsilon(\tau)$  as $\varepsilon\rightarrow 0$, we provide the uniform boundedness of
 these attractors in $\mathbb H^1_\varepsilon(\calo)$ as given below.
\begin{lem}\label{Lue}
Assume that \eqref{cs1}-\eqref{La} and \eqref{acg} hold. Then for every $\tau\in \R$, $0<\varepsilon<\varepsilon_0$
and $\mu^\varepsilon \in \mathcal{A}_\varepsilon(\tau)$,
$$\int_{\mathbb H^1_\varepsilon ( {\mathcal{O}})}\| \xi\|^2_{\mathbb H^1_\varepsilon ( {\mathcal{O}})}\mu^\varepsilon(d\xi)
\le L_3(\tau),$$
where $L_3(\tau)>0$  depends on $\tau$,
but not on $\varepsilon$.
\end{lem}

\begin{proof}
For $\tau\in \R$ and  $0<\varepsilon<\varepsilon_0$,  let $\mu^\varepsilon\in \mathcal{A}_\varepsilon(\tau)$.
By the invariance of $\mathcal{A}_\varepsilon(\tau)$, we get there exist $\{t_n\}_{n=1}^\infty$
with $t_n\rightarrow \infty$ as $n\rightarrow \infty$ and $\mu^\varepsilon_n\in \mathcal{A}_\varepsilon(\tau-t_n)$
such that
\be \label{ae1}
\mu^\varepsilon=S^\varepsilon(t_n,\tau-t_n)\mu^\varepsilon_n.
\ee

Let $\xi^{n,\varepsilon}=(\xi^{n,\varepsilon}_1,\xi^{n,\varepsilon}_2)\in
 L^2_{\mathcal{F}_{\tau-t_n}}(\mathbb  L^2(\mathcal {O}) ))$ with $\mathcal L(\xi^{n,\varepsilon})=\mu^\varepsilon_n$,
By \eqref{w12} we have
\be \label{aw12}
w\left( {\tau,\tau-t_n,\xi ^{n,\varepsilon}  } \right) = w_1 \left( {\tau,\tau-t_n,\xi ^{n,\varepsilon}  } \right)
+ w_2 \left( {\tau,\tau-t_n,\xi ^{n,\varepsilon}  } \right).
\ee
Note that for $\tau\in \R$ and
  $0<\varepsilon<\varepsilon_0$, we get
$$
S^{i,\varepsilon}(t_n,\tau-t_n)\mu^\varepsilon_n
=p_{\ast}^{i,\varepsilon}(\tau,\tau-t_n)
\mu^\varepsilon_n .
$$
By \eqref{v1e} we get
\be \label{as1to}
\E(\|w_1 \left( {\tau,\tau-t_n,\xi ^{n,\varepsilon}  } \right)\|^2)
\rightarrow 0, \quad \text{as}\quad t\rightarrow\infty.
\ee

By Lemma \ref{Tsb4}, \ref{aTsb4} and Chebyshev's inequality,
 we see there exists a large enough $n_0\in \N$ such that
  the sequence $\{p_{\ast}^{2,\varepsilon}(\tau,\tau-t_n)\mu_{n}^\varepsilon\}^\infty_{n=n_0}$
is tight on $\mathbb H^1_w(\calo)$, as balls are weakly compact.
Since weak compacts in $\mathbb H^1(\calo)$ are metrizable,
the set $\{p_{\ast}^{2,\varepsilon}(\tau,\tau-t_n)\mu_{n}^\varepsilon\}^\infty_{n=n_0}$
is relatively sequentially compact in the narrow topology of the space of finite Borel
measures on $\mathbb H^1_w(\calo)$ by \cite[Theorem 6]{Le57}.
Thus there exist a $\mu^\varepsilon_0\in \mathcal P(\mathbb H^1(\calo))$
and a subsequence $\{p_{\ast}^{2,\varepsilon}(\tau,\tau-t_{n_l})\mu_{n_l}^\varepsilon\}^\infty_{l=1}$
 of $\{p_{\ast}^{2,\varepsilon}(\tau,\tau-t_n)\mu_{n}^\varepsilon\}^\infty_{n=n_0}$ such that
\be \label{swc}
\mathop {\lim }\limits_{l \to \infty } \left( {\varphi ,p_{\ast}^{2,\varepsilon}(\tau,\tau-t_{n_l})\mu_{n_l}^\varepsilon } \right)
= \left( {\varphi ,\mu^\varepsilon_0 } \right),\quad \forall\varphi  \in  C_b \left( \mathbb H^1_w(\calo) \right),
\ee
which together with Lemma \ref{Tsb4} and \ref{aTsb4} implies that
\be \label{veh}
\int_{\mathbb H^1_\varepsilon ( {\mathcal{O}})}\| \xi\|^2_{\mathbb H^1_\varepsilon ( {\mathcal{O}})}\mu^\varepsilon_0(d\xi)
\le L_3(\tau).
\ee
It follows from  \eqref{aw12}-\eqref{swc} and  Lemma \ref{Lch} that
 \be \label{vtv}
 \mu^\varepsilon=\mu^\varepsilon_0.
 \ee
Then by \eqref{veh} and \eqref{vtv}
$$
\int_{\mathbb H^1_\varepsilon ( {\mathcal{O}})}\| \xi\|^2_{\mathbb H^1_\varepsilon ( {\mathcal{O}})}\mu^\varepsilon(d\xi)
\le L_3(\tau).
$$
This completes the proof.

\end{proof}

Now, we state the main result of this paper.\\
 {\bf{Proof of Theorem \ref{upsemi}}.}
Given $\tau\in \R$, by  Lemma \ref{Lue} we find that
\begin{equation}\label{ucona1}
\int_{\mathbb H^1_\varepsilon ( {\mathcal{O}})}\| \xi\|^2_{\mathbb H^1_\varepsilon ( {\mathcal{O}})}\mu(d\xi)
\le L_3(\tau) \quad  \mbox{for all } \
  0< \varepsilon <\varepsilon_0 \ \ \mbox{and} \
\mu \in \mathcal{A}_\varepsilon(\tau) ,
\end{equation}
where $L_3(\tau)>0$  depends on $\tau$,
but not on $\varepsilon$.
The proof of the remaining part is similar to the proof of
 Lemma 6.3 in \cite{LW24}, and is therefore omitted.

\end{document}